\documentclass[oneside,english,reqno, 12pt]{amsart}
\usepackage[T1]{fontenc}
\usepackage[latin9]{inputenc}
\usepackage{amsthm}
\usepackage{amssymb}
\usepackage{amsmath}
\usepackage{mathabx}
\usepackage{xcolor}
\usepackage{verbatim, color}
\usepackage{mathtools}
\usepackage{bbm}
\usepackage{float}
\usepackage{dsfont, graphicx}
\usepackage{epstopdf}
\usepackage[hidelinks]{hyperref}
\usepackage[T1]{fontenc}
\usepackage{amssymb,amsmath, amsthm, amsfonts}
\usepackage{graphicx}

\usepackage{thmtools,thm-restate}

\newtheorem{theorem}{Theorem}[section]
\newtheorem{thm}{Theorem}[section]
\newtheorem{lemma}[theorem]{Lemma}
\newtheorem{proposition}[theorem]{Proposition}

\newcommand{\R}{\mathbb{R}}

\newcommand{\stnorm}{L_{\omega}^{\alpha}L_{t}^{\alpha}L_{x}^{\beta}(\Omega\times [0,\infty)\times \mathbb{R}^{3})}

\renewcommand{\d}{\partial}
\newcommand{\f}{\frac}

\newcommand{\beq}{\begin{equation}}
\newcommand{\eeq}{\end{equation}}
\newcommand{\beqq}{\begin{equation*}}
\newcommand{\XXX}{\mathbb{X}}
\newcommand{\eeqq}{\end{equation*}}

\newcommand{\domain}{\Omega\times [0,\infty)\times \mathbb{R}^{3}}
\newcommand{\domaint}{\Omega\times [0,T]\times \mathbb{R}^{3}}

\theoremstyle{definition}

\theoremstyle{remark}
\newtheorem{remark}[theorem]{Remark}

\numberwithin{equation}{section}

\newcommand{\SSS}{\mathcal{S}}

\definecolor{darkgreen}{rgb}{0.1,0.6,0.1}
\definecolor{darkred}{rgb}{0.7,0.1,0.1}
\definecolor{darkblue}{rgb}{0,0,0.7}

\def\scal#1{\langle#1\rangle}

\newcommand{\eps}{\varepsilon}

\begin{document}

\address{Chenjie Fan
\newline \indent Academy of Mathematics and Systems Science and Hua Loo-Keng Key Laboratory of Mathematics, Chinese Academy of Sciences,\indent 
\newline \indent  
Beijing, China.\indent }
\email{fancj@amss.ac.cn}

\address{Weijun Xu
\newline \indent Beijing International Center for Mathematical Research, Peking University, China\indent }
\email{weijunxu@bicmr.pku.edu.cn}

\address{Zehua Zhao
\newline \indent Department of Mathematics and Statistics, Beijing Institute of Technology,
\newline \indent MIIT Key Laboratory of Mathematical Theory and Computation in Information Security,
\newline \indent  Beijing, China. \indent}
\email{zzh@bit.edu.cn}

\title[Long time behavior of stochastic NLS with a small multiplicative noise]{Long time behavior of stochastic NLS with a small multiplicative noise}
\author{Chenjie Fan, Weijun Xu and Zehua Zhao}
\maketitle

\begin{abstract}
We prove the global space-time bound for the mass critical nonlinear Schr\"odinger equation perturbed by a small multiplicative noise in dimension three. The associated scattering behavior are also obtained. We also prove a global Strichartz space-time bound for the linear stochastic model, which is new itself and serves a prototype model for the nonlinear case. The proof combines techniques from \cite{fan2018global}, \cite{fan2020long} as well as local smoothing estimates for linear Schr\"odinger operators.
\end{abstract}

\setcounter{tocdepth}{1}
\tableofcontents

\section{Introduction}
\subsection{The model and statement of main results}
The aim of this article is to study the long-time behavior of the solution $u$ to the equation
\begin{equation}\label{eq: snls}
\begin{cases}
i \d_t u + \Delta u = |u|^{4/3} u + \eps V u \circ \frac{dB_t}{dt}\;, \quad (t,x) \in \R^{+} \times \R^{3}\;,\\
u(0,\cdot)=u_{0} \in L_{\omega}^{\infty} {L^2_{x}}(\mathbb{R}^{3})\;.
\end{cases}
\end{equation}
Here, $V \in \SSS(\R^3)$ is a real-valued Schwartz function independent of time, $B$ is the standard Brownian motion, and $\eps>0$ is a small parameter. Finally, $\circ$ denotes the Stratonovich product which preserves the $L^2$-norm of the solution $u$. This is the stochastic perturbation of mass-critical defocusing nonlinear Schr\"odinger equation in $d=3$. 

We write \eqref{eq: snls} in its It\^o form as
\begin{equation}\label{eq: nonlinearmain}
\begin{cases}
i \d_t u + \Delta u = |u|^{4/3}u - \frac{i \eps^2}{2} V^{2} u  + \eps V u \cdot \frac{d B_t}{d t}\;,\\
u(0,\cdot) \in L_{\omega}^{\infty}L_{x}^{2}\;,
\end{cases}
\end{equation}
where now the product between $u$ and $\frac{dB_t}{dt}$ is It\^o, and the extra term $-\frac{\eps^2}{2}V^{2}u$ is the It\^o-Stratonovich correction, which keeps the $L^2$ norm conserved. 

We will start with a deterministic $L_{x}^{2}$ initial data $u_{0}$, and the extensions to random initial data in $L_{\omega}^{\infty} L_{x}^{2}$ is straightforward as long as the randomness in $u_0$ is independent of the Brownian motion. Our main theorem is the following. 

\begin{thm}\label{thm: main0}
Suppose $V \in \SSS(\R^3)$ is a real-valued, time-independent Schwartz function. Let $\eps$ be sufficiently small (depending on $V$ and $\|u_0\|_{L_x^2}$). Then there exists $C = C(\|u_0\|_{L_x^2})$ such that the solution $u$ to \eqref{eq: nonlinearmain} satisfies the space-time bound
\begin{equation}\label{eq: estimatemain_nonlinear}
\|u\|_{L_{\omega}^{\alpha}L_{t}^{\alpha}L_{x}^{\beta}(\Omega\times [0,\infty)\times \mathbb{R}^{3})} \leq C(\|u_0\|_{L_x^2})\;,
\end{equation}
where $(\alpha, \beta)=(4-, 3+)$ is an admissible pair in $d=3$, and $\alpha\geq \frac{7}{3}$. 

And, we have

\begin{equation}\label{eq: estimatemain_nonlinear222}
\|u\|_{L_{\omega}^{2}L_{t}^{\alpha}L_{x}^{\beta}(\Omega\times [0,\infty)\times \mathbb{R}^{3})} \leq C(\|u_0\|_{L_x^2})\;,
\end{equation}
for all non-end point admissible pair $(\alpha,\beta)$,

Furthermore, we have the scattering asymptotic in the sense there exists $u^{+}\in L_{\omega}^{\infty}L_{x}^{2}$, such that 
\begin{equation}
\|u(t)-e^{it\Delta}u^{+}\|_{L_{\omega}^{\rho^{*}}L_{x}^{2}}\xrightarrow{t\rightarrow \infty} 0,	
\end{equation}
for some $\rho^{*}>1$.
And we also have the a.s. convergence in the sense ,
\begin{equation}
	\|u(t)-e^{it\Delta}u^{+}\|_{L_{x}^{2}}\xrightarrow{t\rightarrow \infty} 0, a.s.
\end{equation}
\end{thm}

\begin{remark}
As in \cite{fan2020long}, the scattering is not a direct consequence of the global space time bound as in the usual deterministic case.	
\end{remark}

\begin{remark}
Here we use the notation $(4-,3+)$ mean $(\alpha,\beta)$ is admissible and 	$\alpha<4$, (which implies $\beta>3$). We don't need $\alpha$ close to $4$. Similarly we use the notation $(2+,6-)$ . 
\end{remark}

We also prove the parallel result for the linear model, which usually bears the name of Strichartz estimates. This is the following theorem.

\begin{thm}\label{thm: main}
Let $u$ satisfy the linear equation (in It\^o form)
\begin{equation}\label{eq: linearmain}
	i \d_t u + \Delta u = \eps V u \cdot \frac{d B_t}{dt} - \frac{i \eps^2}{2} V^{2} u\;, \quad u(0, \cdot) = u_{0}\in L_{x}^{2}\;,
\end{equation}
where $V \in \SSS(\R^3)$ is a real-valued Schwartz function, and $\eps>0$ is a sufficiently small parameter depending on $V$. Then one has the global space-time Strichartz estimate
\begin{equation}\label{eq: estimatemain}
\|u\|_{L_{\omega}^{\alpha} L_{t}^{\alpha}L_{x}^{\beta}(\Omega\times [0,\infty) \times \mathbb{R}^{3})}\lesssim \|u_{0}\|_{L_{x}^{2}}
\end{equation}
where $(\alpha, \beta)=(4-, 3+)$ and $(2+,6-)$ is an admissible pair for the free Schr\"odinger operator. 
\end{thm}

\begin{remark}
The scattering behavior for the linear equation has already been presented in \cite{FX}.
\end{remark}
\begin{remark}
For both the nonlinear and linear models,  one can use interpolation with pathwise $L^{2}$ conservation law to recover $(\alpha,\beta)$ in the range from $(4,3)$ to $(\infty, 2)$. However, our analysis misses the endpoint $(2,6)$ even in the linear case. It appears that besides the usual importance in the deterministic case, the pair $(2,6)$ is also of particular interest in our stochastic problem, see Subsection \ref{sec: bri}.  We also point out that in the linear case, the smallness of $\eps$ depends on $V$, but not on the size of initial data $\|u_0\|_{L_x^2}$. 
\end{remark}

\begin{remark}
Since the equation \eqref{eq: linearmain} is linear, Theorem~\ref{thm: main} also extends to random initial data $u_0 \in L_{\omega}^{\alpha}L_{x}^{2}$ immediately, as long as the randomness is independent with the Brownian motion. 
\end{remark}

\begin{remark}
Some key ingredients in this article is of perturbation nature, and hence cannot be applied to the large noise case, in particular for the nonlinear model. On the other hand, it is not hard to generalize to replace the noise $V \frac{d B_t}{d t}$ by infinite dimensional Wiener process as long as there is enough spatial regularity. But since our main focus is the effect of the noise on the long-time behavior of the dynamics, we focus on the simplest possible case and consider one Brownian motion. 
\end{remark}

\subsection{Background and motivation}

The main aim of this article is to investigate the effect of a multiplicative noise (without decay in time) in the long time behavior of a dispersive system. 

The nonlinear Schr\"odinger equation is one of the most typical dispersive equations. The global well-posedness of the $d$-dimensional deterministic defocusing nonlinear Schr\"odinger equation
\begin{equation} \label{eq: deterministic_nls}
    i \d_t v + \Delta v = |v|^{p-1} v\;, \qquad v(0,\cdot) \in L_x^2(\R^d)
\end{equation}
with $L_x^2$ initial data has been an important topic of study in dispersive equations. The equation is called \textit{mass-subcritical} when $p-1 < \frac{4}{d}$, and \textit{mass-critical} when $p-1 = \frac{4}{d}$. The local well-posedness to \eqref{eq: deterministic_nls} for $p \in [1,1+\frac{4}{d}]$ is based on Strichartz estimates and is well understood classically. When $p < 1 + \frac{4}{d}$, the local existence time depends on the initial data via its $L^2$-norm only, so global well-posedness of the solution is a direct consequence of the local well-posedness and $L^2$ conservation law. 

When $p=1+\frac{4}{d}$ (which corresponds to the mass-critical case), the global existence of the solution to \eqref{eq: deterministic_nls} with general $L^2$ initial data becomes much subtler. It has been outstanding for a long time, and was finally resolved by Dodson in a series of works \cite{Dodson3, Dodson2, Dodson1}. Furthermore, the global space-time bound of the solution in terms of the $L^2$-norm of the initial data was established, and scattering of the solution comes as a consequence. 

The stochastic version of \eqref{eq: deterministic_nls} with a Stratonovich multiplicative noise is given by
\begin{equation} \label{eq: stochastic_nls}
	i \d_t u + \Delta u = |u|^{p-1} u + u \circ \frac{d W_t}{dt}\;, 
\end{equation}
where $W$ is typically a Gaussian noise that is white in time and colored in space. This model has been studied under various spatial assumptions on $W$, most of which could roughly be formulated by
\begin{equation*}
	W_{k}(t,x) = \sum_{k} \lambda_k B_{t}^{(k)} V_{k}(x)\;,
\end{equation*}
where $B^{(k)}$ are independent standard Brownian motions, $V_k$ are nice enough functions, and $\lambda_k$ satisfies proper decay assumptions. To the best of our knowledge, \cite{dBD99} was the first to construct a global solution to \eqref{eq: stochastic_nls} for $p < 1 + \frac{4}{d}$ (\textit{mass-subcritical} case). Subsequent refinements and extensions (also in subcritical cases) include \cite{dBD03}, \cite{BRZ14, Brzezniak-Millet, Brzezniak-Liu-Zhu}, etc. 

In \cite{fan2018global}, the first two authors established global well-posedness for the stochastic \textit{mass-critical} case in dimension $1$ ($p = 5$). The main ingredients are careful perturbation analysis around the global well-posedness results of Dodson as well as a martingale type control. Hence the arguments also work for other dimensions, at least when the nonlinearity is not too irregular. \cite{DZhang} independently constructed global solutions to the stochastic defocusing equations in both mass and energy critical regimes. 

On the other hand, very little was known about the long time behavior of the solutions. As a toy model, one may introduce a parameter $\gamma>0$, and consider the case where the noise $W$ in \eqref{eq: stochastic_nls} is given by
\begin{equation*}
	W(t,x) = \sum_{k} \lambda_k V_{k}(x) \cdot \frac{B_{t}^{(k)}}{\scal{t}^{\gamma}}.
\end{equation*}
When $\gamma>1$, the time-decay of the noise is strong enough, and it is not hard to see that the long time dynamics can be reduced to local dynamics. From the viewpoint of Burkholder inequality, the next natural threshold is $\gamma>\frac{1}{2}$, in which case the noise $W$ has finite quadratic variation over the whole time line. Indeed, in \cite{HRZ}, the authors established the scattering of the solution to \eqref{eq: stochastic_nls} for noise with finite quadratic variation over the whole timeline\footnote{Some extra technical assumptions are needed, but we focus on the time decay property of the noise here.}, which covers all $\gamma > \frac{1}{2}$. But their techniques seem to break down at $\gamma = \frac{1}{2}$, where now the noise has infinite quadratic variation. In the recent work \cite{fan2020long}, the first and third authors established a global space-time bound and scattering of the solution to \eqref{eq: stochastic_nls} in dimension $3$ with $p-1=\frac{4}{3}$ (critical in dimension $3$) and arbitrary $\gamma>0$. The techniques also work for general $d>3$. 

The natural question now is whether one can set the factor $\gamma=0$, not allowing asymptotic decay of the noise in time, and this is the main aim of this article (see also \cite{FX} for the dispersive estimate for the \textit{linear} equation with a small multiplicative constant $\eps$ in front of the noise but no decay assumption). 

Our main result is to finally establish the global space-time bound  and scattering for \eqref{eq: stochastic_nls} without a time-decay factor in the noise (but with a fixed small multiplicative constant $\eps$). We want to remark that although the main interest here is to understand long time (decaying) behaviours of the solution when the noise itself does not decay, our theorem also extends the range of exponents (in solution space) allowed for local well-posedness compared to the previous study in \cite{fan2018global}\footnote{\cite{fan2018global} is written for 1d model only, but can be naturally generalized to other dimensions}, where the spatial integrability exponent $\beta$ could only take value below $3$. 

As a by-product of the proof, we also establish a global space-time bound for the linear stochastic equation \eqref{eq: linearmain}, which is a global-in-time stochastic Strichartz estimate. Although there have been various types stochastic Strichartz estimates (for example \cite{dBD99, Brzezniak-Millet, Hor, Brzezniak-Liu-Zhu}), all of them are on a fixed time interval. To the best of our knowledge, Theorem~\ref{thm: main} is the first global-in-time one. 

We refer to Sections~\ref{sec: overview_linear} and~\ref{sec: overview_nonlinear} for a description of the strategy in proving Theorems~\ref{thm: main} and~\ref{thm: main0}. 

Finally, we want to refer to the introduction of \cite{dBD03}, \cite{HRZ} for the physical background of \eqref{eq: nonlinearmain}. But if you believe NLS is of physical importance, and there is always some noise in the real world and measurement, then the study of \eqref{eq: nonlinearmain} is of course natural.

\subsection{A brief discussion on $TT^{*}$ and endpoint Strichartz estimate}\label{sec: bri}

We give a brief discussion on our linear estimate \eqref{eq: estimatemain}. Strichartz estimates play a fundamental role in the study of dispersive equations. Historically, deterministic Strichartz estimates follow from the dispersive estimates via a $TT^{*}$ argument. For our model \eqref{eq: linearmain}, the corresponding dispersive estimate was proved in \cite{FX}. But defining the operator $T^{*}$ requires running Brownian motion and solving the equation backwards in time. Since the Burkholder inequality and martingale structure is used essentially in our proof, it is not clear for us at this stage on how to apply similar scheme to the nonlinear model, but we expect rough path theory might be of help here. On the other hand, our proof for the linear problem here also serves as a toy model for the nonlinear problem. 

We also remark again that we miss the endpoint $(2,6)$ in our Strichartz estimate. This end point is highly non-trivial even in the free Schr\"odinger case (see \cite{keel1998endpoint}). One key idea in \cite{keel1998endpoint} is to treat certain dual estimate as bi-linear estimates rather than linear ones. 

The endpoint $(2,6)$, except its usual importance in deterministic problem, also has particular interest our multiplicative noise model. This is because from view point of Burkholder inequality, it is of interest to establish $L_{t}^{2}$ time bound for $V u dB_{t}$ (this is also why the finite quadratic variation condition is important in the work of \cite{HRZ}). We hope to be able to treat the end point in future works. 

\begin{remark}
One can see from the proof, in particular in the last section,  that both in the linear and nonlinear case, we can prove $u$  can be splitted into several parts, all satisfying an $L_{t}^{2}L_{x}^{6}$ part except for a term which is of form $e^{it\Delta}f(t)$, where $f(t)$ is a martingale in $L_{x}^{2}$ which converges to some $f\in L_{\omega}^{\infty}L_{x}^{2}$. We expect this may be of some help but it is not clear for us whether this enough to conclude the $L_{t}^{2}L_{x}^{6}$ bound.
\end{remark}

\subsection{Notations}

Throughout the article, we fix the potential $V \in \SSS(\R^3)$ and $\eps$ sufficiently small depending on $V$ (and also depending on $\|u_0\|_{L_x^2}$ in the nonlinear model). We will fix a small $\delta>0$ throughout the whole article, which appear in local smoothing type estimates. We write $S(t) = e^{it\Delta}$ for the propagator of the free Schr\"odinger equation, and $H(t)$ for the propagator of the linear damped equation; that is, $H(t) \psi_0$ solves the equation
\begin{equation*}
	i \d_t \psi + \Delta \psi = \frac{i \eps^2}{2} V^2 \psi\;, \qquad \psi(0,\cdot) = \psi_0\;.
\end{equation*}
Note that $H(t) = H_{\eps}(t)$ depends on the parameter $\eps$. But since $V$ and $\eps$ are fixed throughout the article, we will omit it in the notation and simply write $H(t)$. We also write $\scal{x} = (1+|x|^2)^{\f{1}{2}}$, and $p'$ for the dual index of $p \in (1,+\infty)$ in the sense that $\frac{1}{p^{'}}+\frac{1}{p}=1$. 

We write $A \lesssim B$ to say that there is a constant $C$ such that $A\leq CB$. We use $A \simeq B$ when $A \lesssim B \lesssim A $. Particularly, we write $A \lesssim_u B$ to express that $A\leq C(u)B$ for some constant $C(u)$ depending on $u$. Without special clarification, the implicit constant $C$ can vary from line to line.

\subsection{Structure of the paper}

The rest of the article is organized as follows. In Section~\ref{sec: prelim}, we state some preliminary results, including the classical Strichartz and local smoothing estimates for the free Schr\"odinger operator and the Burkholder inequality for martingales in proper Banach spaces. In Sections~\ref{sec: overview_linear} and~\ref{sec: overview_nonlinear}, we give an overview of the strategies in proving Theorems~\ref{thm: main} and~\ref{thm: main0}. Section~\ref{sec: md} is devoted to establishing properties of the damped Schr\"odinger operator $H$, including the corresponding dispersive, Strichartz, and local smoothing estimates. In Section~\ref{sec: linearmain}, these estimates for the damped operator are used to establish Theorem~\ref{thm: main}. Finally, in Section~\ref{sec: mm}, we prove the main nonlinear estimates and establish Theorem~\ref{thm: main0}.

\subsection*{Acknowledgement}

Fan was partially supported in NSFC grant No.11688101. Xu was partially supported by the National Key R\&D Program of China (No.2020YFA0712900). Zhao was partially supported by the NSF grant of China (No. 12101046) and the Beijing Institute of Technology Research Fund Program for Young Scholars.

\section{Preliminaries} \label{sec: prelim}

\subsection{Dispersive estimates  and Strichartz estimates for $e^{it\Delta_{\mathbb{R}^d}}$}
We state the standard decay estimate (also known as dispersive estimate) and Strichartz estimate for Schr\"odinger operator $e^{it\Delta_{\mathbb{R}^d}}$. We refer to \cite{cazenave2003semilinear,tao2006nonlinear} for details. 
\begin{lemma}[Dispersive estimate]
For the linear propagator $e^{it\Delta_{\mathbb{R}^d}}$ of Schr\"odinger equation in $\R^{d}$, one has 
\begin{equation}\label{eq: dispersive}
\|e^{it\Delta_{\mathbb{R}^d}}u_{0}\|_{L_{x}^{\infty}}\lesssim t^{-\frac{d}{2}}\|u_{0}\|_{L_{x}^{1}}.
\end{equation}
Moreover, by interpolation with the mass conservation, for $p\geq 2$, we have
\begin{equation}\label{lem: dispersive}
\|e^{it\Delta_{\mathbb{R}^d}}u_{0}\|_{L_{x}^{p}}\lesssim t^{-d(\frac{1}{2}-\frac{1}{p})}\|u_{0}\|_{L_{x}^{p'}}.
\end{equation}
\end{lemma}

\begin{lemma}[Strichartz estimate]\label{lem: strichartz} \label{Euclidean}
Suppose $\frac{2}{q}+\frac{d}{p}=\frac{d}{2}$, where $p,q \geq 2$ and $(q,p,d)\neq (2,\infty,2)$. We call such pair $(p,q)$ a Strichartz pair. Then 
\begin{equation}
    \|e^{it\Delta_{\mathbb{R}^d}}f\|_{L^q_tL^p_x(\mathbb{R}\times \mathbb{R}^d)} \lesssim \|f\|_{L^2(\mathbb{R}^d)}.
\end{equation}
Also, for any two Strichartz pairs $(p_1,q_1)$ and $(p_2,q_2)$, we have
\begin{equation}
    \Big\| \int_0^t e^{i(t-s)\Delta_{\mathbb{R}^d}}F(s)ds \Big\|_{L^{q_1}_tL^{p_1}_x(\mathbb{R}\times \mathbb{R}^d)} \lesssim \|F\|_{L^{q_{2}^{'}}_{t} L^{p_{2}^{'}}_x(\mathbb{R}\times \mathbb{R}^d)}.
\end{equation}
\end{lemma}

We fix $d=3$ throughout the article.

\subsection{Local smoothing estimates for $e^{it\Delta}$}

We state the standard local smoothing estimate as follows (see \cite{constantin1987effets,vega1988schrodinger,sjolin1987regularity} and also \cite{tao2006nonlinear} for details).
\begin{lemma}\label{lem: sls1}
\begin{equation}\label{eq: sls1}
\int \int \langle x \rangle^{-1-\delta}|\nabla^{1/2} e^{it\Delta}f|^{2} dx dt \lesssim_{\delta} \|f\|_{L_{x}^{2}}^{2}.
\end{equation}
\end{lemma}

Note that \eqref{eq: sls1} implies the associated estimate for the Duhamel part, via the standard $TT^{*}$ argument,
\begin{equation}\label{eq: sls1cor}
\Big\| \langle x \rangle^{-\frac{1}{2}-\delta}\nabla\int_{0}^{t}e^{i(t-s)\Delta}F (s) ds \Big\|_{L_{t}^{2}L_{x}^{2}}\lesssim_{\delta}\|\langle x \rangle^{\frac{1}{2}+\delta}F\|_{L_{t,x}^{2}}.
\end{equation}

Next, we state a pointwise version of local smoothing estimate. 

\begin{lemma}\label{lem: sls2}
	For every $R>0$ and $t \geq 1$, we have
\begin{equation}\label{eq: sls2}
\|\nabla e^{it\Delta}f\|_{L_{x}^{2}(B_{R})}\lesssim \frac{R}{t}\|\langle x \rangle f\|_{L_{x}^{2}}\;.
\end{equation}
\end{lemma}

Lemma~\ref{lem: sls2} follows from the conservation of conformal energy $\|(x+it\nabla)u\|_{L_{x}^{2}}^{2}$. See \cite{tao2006nonlinear} for a detailed proof.

\subsection{GWP for mass critical nonlinear Schr\"odinger equations}

Consider defocusing mass critical NLS in $\mathbb{R}^{3}$,
\begin{equation}\label{eq: nls}
\begin{cases}
iu_{t}+\Delta u=|u|^{4/3}u,
u(0,x)=u_{0}\in L_{x}^{2}.
\end{cases}
\end{equation}

We recall the following seminal result by Dodson,
\begin{thm}\label{thm: dodson}
Let $u$ solves \eqref{eq: nls} with initial data $u_{0}\in L_{x}^{2}$, then $u$ is global with
\begin{equation}
\|u\|_{L_{t}^{2}L_{x}^{6}\cap L_{t}^{\infty}L_{x}^{2}}\lesssim C_{\|u_{0}\|_{L_{x}^{2}}}.	
\end{equation}

Furthermore, $u$ scatter to a linear (free Schr\"odinger) solution. 
\end{thm}
See \cite{Dodson3}.

\subsection{The Burkholder inequality}
We will frequently use the following version of Burkholder inequality (for martingales in Banach space).
\begin{lemma}\label{lem: burk}
	Let $B$ be the standard Brownian motion. Let $\rho \in (1, +\infty)$, $p \in [2,+\infty)$, and $\Phi$ be an $L^p(\R^d)$-valued process adapted to the filtration generated by $B_t$. Then, we have
\begin{equation}\label{eq: burk}
\sup_{0 \leq a \leq b \leq t} \bigg\| \int_{a}^{b} \Phi(s) d B_{s} \bigg\|_{L_{\omega}^{\rho}L_{x}^{p}} \lesssim_{p,\rho} \bigg\| \int_{0}^{t} \|\Phi(s)\|_{L_{x}^{p}}^{2} d s \bigg\|_{L_{\omega}^{\rho/2}}^{1/2}\;.
\end{equation}
\end{lemma}

We refer to \cite{BDG73, Bur73, B, BP} for proofs of the statement and also statements in more general situations.

\section{Overview of the proof of Theorem \ref{thm: main} and some reductions} \label{sec: overview_linear}

Though it is natural to write down the Duhamel formula for solution $u$ to \eqref{eq: linearmain} based on linear Schr\"odinger equations and view $u\circ dW_{t}$ as an input,
\begin{equation}\label{eq: Duhamel1}
u(t)=e^{it\Delta}u_{0}-i\int_{0}^{t}e^{i(t-s)\Delta}(Vu(s))dB_{s}-\frac{1}{2}\int_{0}^{t}e^{i(t-s)\Delta}V^{2}uds\;,
\end{equation}
it turns out, as observed in \cite{fan2020long}, it may be of advantage\footnote{Note that \cite{fan2020long} studies a nonlinear model with a slow decaying noise while now we are study a linear model, with noise does not decay in time.} to view \eqref{eq: linearmain} as a damped Schr\"odinger equation perturbed via the term $udW_{t}$.   Let $H(t-s)$ be the linear propagator for linear damped Schr\"odinger equation.
\begin{equation}\label{eq: dls}
iv_{t}+\Delta v=-i\eps^{2} V^{2}v
\end{equation}
i.e. $v:=H(t-s)g$ solves equation \eqref{eq: dls} with $v(s)=g$.

And we write down the Duhamel formula for $u$ as 

\begin{equation}\label{eq: Duhamel2}
u(t,x) = H(t)u_{0} - i \eps \int_{0}^{t}H(t-s)(V u(s))dB_{s}.
\end{equation}
For notational convenience, we write
\begin{equation}\label{eq: gamma}
(\Gamma u)(t):=-i\int_{0}^{t}H(t-s)(Vu(s))dB_{s},
\end{equation}
and \eqref{eq: Duhamel2} now reads
\begin{equation}
u(t) = H(t)u_{0}+\eps (\Gamma u)(t).
\end{equation}
We will focus on a priori estimates, and the desired estimate \eqref{eq: estimatemain} will follow via a standard continuity argument when $\eps$ is chosen small enough. Since Strichartz estimates for $H$ in \eqref{eq: strichartzh} well-handles the term $H(t)u_{0}$, we are left with the following proposition (as an a priori estimate).

\begin{proposition}\label{prop: linearmain}
When $\eps$ is chosen small enough. Let $u$ solve \eqref{eq: linearmain} with initial data $u_{0}\in L_{x}^{2}$, then one has 
\begin{equation}\label{eq: estimateprop}
\|\Gamma u(t)\|_{L_{\omega}^{\alpha}L_{t}^{\alpha}L_{x}^{\beta}(\Omega\times [0,\infty)\times \mathbb{R}^{3})}\lesssim \|u_{0}\|_{L_{x}^{2}}+\|u\|_{L_{\omega}^{\alpha}L_{t}^{\alpha}L_{x}^{\beta}(\Omega\times [0,\infty)\times \mathbb{R}^{3})}.
\end{equation}
\end{proposition}

\begin{remark}
We point out that $H$ depends on $\eps$,  and our proof relies on some deterministic estimates for the propagator  $H$, which is only proved for small value of $\eps$ in this article\footnote{However,  we also note it is possible to establish the same estimates for more general $H$, and in particular does not require a smallness assumption, we don't fully explore this here.}.
\end{remark}

The proof of \eqref{eq: estimatemain} relies on the fact that $u$ is a solution to \eqref{eq: linearmain}. Indeed,  we need to do one more expansion of \eqref{eq: gamma} to get the expression
\begin{equation}\label{eq: gammaworking}
\begin{aligned}
\Gamma u =& -i \int_{0}^{t} H(t-s) (V H(s) u_{0}) dB_{s}\\
&+ \eps\int_{0}^{t} H(t-s) \left( V \int_{0}^{s} \big( H(s-r) V u(r) \big) dB_{r} \right) dB_{s}.
\end{aligned}
\end{equation}

\begin{remark}\label{rem: maxiaml}
Since all the terms in \eqref{eq: gammaworking} are  stochastic integrals and will be estimated with the help of Burkholder inequality, our estimate indeed gives a stronger version of \eqref{eq: estimateprop}
\begin{equation}
\bigg\| \sup_{a\leq b\leq t}\|\int_{a}^{b} H(t-s) (Vu(s)) dB_{s}\|_{L_{x}^{\beta}} \bigg\|_{L_{\omega}^{\alpha}L_{t}^{\alpha}(\Omega\times [0,\infty) )} \lesssim \|u_{0}\|_{L_{x}^{2}}+\|u\|_{L_{\omega}^{\alpha}L_{t}^{\alpha}L_{x}^{\beta}(\Omega\times [0,\infty)\times \mathbb{R}^{3})}.
\end{equation}
Such maximal type estimates play an important role in nonlinear problems (\cite{fan2018global, fan2019wong, fan2020long}).
\end{remark}

The first term in \eqref{eq: gammaworking} will be estimated by the following lemma. 

\begin{lemma}\label{lem: linearesmin}
When $\eps$ is small enough, one has 
\begin{equation}\label{eq: es0}
\Big\| \int_{0}^{t}H(t-s)(V H(s)u_{0})dB_{s} \Big\|_{L_{\omega}^{\alpha}L_{t}^{\alpha}L_{x}^{\beta}(\Omega\times [0,\infty)\times \mathbb{R}^{3})}\lesssim \|u_{0}\|_{L_{x}^{2}}.
\end{equation}		
\end{lemma}

We will see later that the above lemma should be understood as the left hand side of \eqref{eq: es0} being controlled by free linear solution $H(t)u_{0}$.

To estimate the second term in \eqref{eq: gammaworking}, we introduce more notations. Let 
\begin{equation}
A(t):=V\int_{0}^{t}H(t-s) (Vu(s))dB_{s}. 
\end{equation}
Hence, we have
\begin{equation}
\int_{0}^{t}\int_{0}^{s} H(t-s) \Big( V H(s-r) \big( Vu(r) \big) \Big)dB_{r}dB_{s}=\int_{0}^{t}H(t-s)A(s)dB_{s}.
\end{equation}

The estimates for the second term in \eqref{eq: gammaworking} will be the material of the following lemma

\begin{lemma}\label{lem: linearesmain}
We have the following two estimates for $A(t)$: 
\begin{equation}\label{eq: lineara1}
\|A(t)\|_{L_{\omega}^{\alpha}L_{t}^{\alpha}W_{x}^{s_{\alpha},\tilde{\beta}'}(\domain)}\lesssim \|u\|_{L_{\omega}^{\alpha}L_{t}^{\alpha}L_{x}^{\beta}(\domain)}, 
\end{equation}
where $s_{\alpha}=\frac{1}{\alpha}$, $\frac{3}{\tilde{\beta}}=\frac{3}{2}-\alpha$, and
\begin{equation}\label{eq: lineara2}
\|A(t)\|_{L_{\omega}^{\alpha}L_{t}^{\alpha}L_{x}^{\beta'}(\domain)}\lesssim \|u\|_{L_{\omega}^{\alpha}L_{t}^{\alpha}L_{x}^{\beta}(\domain)}.
\end{equation}
Furthermore, we have the following two estimates
\begin{equation}\label{eq: linearlocalpart}
\bigg\| \int_{t-1\vee 0}^{t}H(t-s)A(s)dB_{s} \bigg\|_{\stnorm}\lesssim \|A(t)\|_{L_{\omega}^{\alpha}L_{t}^{\alpha}W_{x}^{s_{\alpha},\beta'}(\domain)}
\end{equation}
and 
\begin{equation}\label{eq: linearglobalpart}
\bigg\| \int_{0}^{t-1\vee 0}H(t-s)A(s)dB_{s} \bigg\|_{\stnorm}\lesssim \|A(t)\|_{L_{\omega}^{\alpha}L_{t}^{\alpha}L_{x}^{\beta'}(\domain)}.
\end{equation}
\end{lemma}

We refer to Remark \ref{rem: num} for the natural appearance of $s_{\alpha}$ and $\tilde{\beta}$. We will prove Lemmas~\ref{lem: linearesmin} and~\ref{lem: linearesmain} in Section~\ref{sec: linearmain}.

We end this section with an explanation why the admissible pair $(4,3)$ is special in our analysis, and how local smoothing type estimates may help us handle $(4-,3+)$. It is very tempting to directly prove, based on \eqref{eq: gamma} and assumptions $u\in \stnorm$, that 
\begin{equation}\label{eq: toy1}
\int_{0}^{t} H(t-s) (Vu(s)) dB_{s} \in \stnorm.
\end{equation}
From the view point of Burkholder Inequality, one has 
\begin{equation}\label{eq: toy2}
\bigg\| \int_{0}^{t}H(t-s) (V u(s)) dB_{s} \bigg\|_{L_{x}^{\beta}}\sim \left(\int_{0}^{t}\|H(t-s)Vu\|^{2}_{L_{x}^{\beta}}\right)^{1/2}.
\end{equation}
Let us neglect the integrability issue in the $\omega$ variable for the moment, the question now becomes the following: for $F(t,x)\in L_{t}^{\alpha}L_{x}^{\beta}([0,\infty)\times \mathbb{R}^{3})$, can one prove
\begin{equation}\label{eq: toy3}
 \int_{0}^{t}\|H(t-s)VF\|^{2}_{L_{x}^{\beta}} ds \; \in L_{t}^{\alpha/2}([0,\infty))?
\end{equation}
Using the fact that $V$ is well localized and $H(t-s)$ satisfies the same dispersive estimate as $e^{it\Delta}$, Lemma \ref{lem: dispersiveh}, one may plug in
\begin{equation}\label{eq: toy4}
\|H(t-s)VF\|^{2}_{L_{x}^{\beta}} \lesssim t^{-\frac{6}{2}+\frac{6}{\beta}}\|F(t)\|_{L_{x}^{\beta}}^{2},
\end{equation}
into \eqref{eq: toy3}.

When $(\alpha,\beta)=(4,3)$, one get a logarithmic divergence since one is now estimating
\begin{equation}
\|\int_{0}^{t}(t-s)^{-1}\|F(s)\|_{L_{x}^{\beta}}^{2}\|_{L_{t}^{\alpha/2}}.
\end{equation}

Given $F\in L_{t}^{\alpha}L_{x}^{\beta}$, one observes that $(t-s)^{-1}$ has a divergence both at $s$ close to $t$ and $s$ far away with $t$.

By choosing $(\alpha, \beta)=(4-,3+)$, one overcomes the divergence for $s$ far away $t$, but paying the price that $(t-s)^{-3+\frac{6}{\beta}}$ is more singular at $|t-s|\ll 1$.

Now, the idea is, $F(s)$ is supposed to be like $H(s)u_{0}$, (at least at first approximation), and $V(x)H(s)u_{0}$ can be raised a half derivative via local smoothing argument, Lemma \ref{lem: ls1}, and one may try to estimate, schematically, for $s$ close to $t$
\begin{equation}\label{eq: toy5}
	\|H(t-s)VF\|^{2}_{W_{x}^{\mu_{\beta},\tilde{\beta}}}\lesssim (t-s)^{-4+\frac{6}{\tilde{\beta}}}\|F\|_{L_{x}^{\beta}}^{2}.
\end{equation}
This can solve the problem for non-integrability of $(t-s)^{-3+\frac{6}{\beta}}$ for $s$ close to $t$, as far as $\tilde{\beta}=3-$ and $W^{\mu_{\beta}, \tilde{\beta}}$ embeds $L_{x}^{\beta}$.

Estimate \eqref{eq: toy5} is not quite true\footnote{At least, we do not know how to prove it.}, (since one only obtain a \textbf{local} smoothing), thus we need to expand as in \eqref{eq: gammaworking} and some other technical problems will arise and need to be handled. But this is the main idea.

\section{Overview of the proof of Theorem \ref{thm: main0}} \label{sec: overview_nonlinear}

We now turn to an overview of the proof of our main result, Theorem \ref{thm: main0}.  We will mainly focus on \eqref{eq: estimatemain_nonlinear}. Given global space time bound, \eqref{eq: estimatemain_nonlinear}, 
the scattering behavior will follow by refining the analysis in \cite{fan2020long},  and  we will explain that in Section \ref{sec: scattering}.

 And we will explain how to recover all the other non-end point case, \eqref{eq: estimatemain_nonlinear222} in the last Section of this article.

We start with an explanation of overall strategy.
\subsection{The strategy}
Recall in the linear case, \eqref{eq: linearmain},  as discussed in the previous section, the goal is to estimate $\int_{0}^{t}H(t-s)VudB_{s}$. Via 
Duhamel Formula, we split $u$ into the summation of $H(t)u_{0}$ and $\int_{0}^{t}H(t-s)VudB_{s}$,  and estimate $\int_{0}^{t}H(t-s)VH(s)u_{0}dB_{s}$ and $\int_{0}^{t}H(t-s)\left(\int_{0}^{s}VH(s-r)dB_{r}\right)dB_{s}$ separately. 

In the nonlinear case, for solution $u$ to \eqref{eq: nonlinearmain}, following \cite{fan2018global}, \cite{fan2020long}, we want to estimate a maximal version of $\int_{0}^{t}H(t-s)VudB_{s}$ rather than itself. Recall Remark \ref{rem: maxiaml}, we do not worry too much to strengthen our estimates to maximal type.

It is now tempting to apply the Duhamel formula for the nonlinear equation, \eqref{eq: nonlinearmain}, and to split $u$ into three parts, $H(t)u_{0}$, $\int_{0}^{t}H(t-s)|u|^{4/3}u ds$, $\int_{0}^{t}H(t-s)VudB_{s}$, and  estimate them separately.  Such an approach, thorough natural, is not well suited for our problem. Indeed, \textbf{the key extra idea} here is, we will indeed split $u=u_{1}+u_{2}$ such that $u_{1}$ will behave like a linear solution, and in particular allows an application of local smoothing type estimate. And $u_{2}$ behaves as a maximal version of  $\int_{0}^{t}H(t-s)VudB_{s}$, so $u_{2}$ can also be estimated in a bootstrap scheme since we are working on a small noise case. We note such a decomposition\footnote{However, the property of such a decomposition was not fully exploited to suit the purpose of current article.} indeed already appear in \cite{fan2020long} for the study of scattering dynamic, Lemma 3.6, also implicitly appear in \cite{fan2018global} for a study of local theory. The existence of such a decomposition is not trivial and not that straightforward. We now go to more details.

\subsection{Main technical reductions and bootstrap scheme}
We will fix $V$ and $m_{0}:=\|u\|_{L_{x}^{2}}$. We will assume $\eps$ is small enough so that for any universal constant $C$, (which may implicitly depending on $m_{0}$ and $V$), $C\eps\ll 1$.  We also fix an admissible pair $(\alpha,\beta)=(4-,3+)$. We also define $s_{\alpha}=\frac{1}{\alpha}$,  $-s_{\alpha}+\frac{3}{\tilde{\beta}}=\frac{3}{\beta}$ based on considerations in Remark \ref{rem: num}.

Let $u$ be the solution to \eqref{eq: snls}, it is (of course) crucial that all the implicit constants involved in the analysis below cannot depend on $u$, (except that one has a priori, by pathwise mass conservation law, $\|u\|_{L_{\omega}^{\infty}L_{t}^{\infty}L_{x}^{2}}\leq \|u_{0}\|_{L_{x}^{2}}$.)

We will introduce, in the same spirit of \cite{fan2020long}, (which is a refinement of \cite{fan2018global}),
\begin{equation}\label{eq: maximal}
\begin{aligned}
M^{*}_{\eps}(t):= &\sup_{a\leq b\leq t} \Big\|\int_{a}^{b}H(t-s)\eps V(x)u(s)dB_{s} \Big\|_{L_{x}^{\beta}}\\
+&\sup_{a\leq b\leq t} \Big\|\int_{a}^{b}\langle x \rangle^{-100}H(t-s)\eps V(x)u(s)dB_{s} \Big\|_{W_{x}^{s_{\alpha},\tilde{\beta}}}.
\end{aligned}
\end{equation}
Compared with \cite{fan2020long}, there are two major difference.
We choose $(\alpha, \beta)=(4-,3+)$ rather than $(4+, 3-)$. We add a second part in $M^{*}$ from considerations on local smoothing estimate.\\

We will denote $M^{*}_{\eps}(t)$ as $M^{*}(t)$ whenever there is no confusion. Note that $M^{*}(t)$ is a random variable, depending on $\omega$, (as the solution $u$), though we usually don't write out the $\omega$ explicitly in this article. We know the following proposition from \cite{fan2020long}. 

\begin{proposition}\label{prop: nlmain1}
Let $u$ be the solution to \eqref{eq: snls}. Then one has
\begin{equation}
\|u\|_{L_{\omega}^{\alpha}L_{t}^{\alpha}L_{x}^{\beta}(\Omega\times [0,T]\times \mathbb{R}^{3})}\lesssim_{m_{0}} 1+\|M^{*}(t)\|_{L_{\omega}^{\alpha}L_{t}^{\alpha}([0,T]\times\mathbb{R}^{3})},	
\end{equation}
uniformly over $T>0$. 
\end{proposition}

One may refer to \cite{fan2020long}, but we will indeed prove a stronger version of Proposition~\ref{prop: nlmain1} in the current article (which also covers this Proposition); see Proposition~\ref{prop: nlmain2} below and also \eqref{eq: working}.

So the goal is to prove
\begin{equation}\label{eq: keysum}
\|M^{*}(t)\|_{L_{\omega}^{\alpha}L_{t}^{\alpha}(\Omega\times [0,\infty))}<\infty	
\end{equation}
when $\eps$ is small enough, and then Theorem \ref{thm: main0} will follow.

From view point of standard continuity and bootstrap argument we will focus on { a priori } estimates for global solution $u$ to \eqref{eq: snls}. (Strictly speaking one need to do bootstrap for all solution $u$ in $[0,T]$, prove some uniform estimates in $T$, and then extend $T$ to infinity, we will only focus on the limit case $T=\infty$ for clarity, and leave the other details for interested readers, but the arguments for $T$ are essentially same.)

We will prove, \textbf{assuming} $\|M^{*}(t)\|_{L_{\omega}^{\alpha}L_{t}^{\alpha}(\Omega\times [0,\infty))}$, that
\begin{equation}\label{eq: keyapriori}
	\|M^{*}(t)\|_{L_{\omega}^{\alpha}L_{t}^{\alpha}(\Omega\times [0,\infty))}\lesssim 1+\eps \|M^{*}(t)\|_{L_{\omega}^{\alpha}L_{t}^{\alpha}(\Omega\times [0,\infty))}.
\end{equation}
 Note that when $\eps$ is small enough, one derives the desired bootstrap type estimate
\begin{equation}
	\|M^{*}(t)\|_{L_{\omega}^{\alpha}L_{t}^{\alpha}(\Omega\times [0,\infty))}\lesssim 1.
\end{equation}
Strictly speaking one need to do bootstrap for all solution $u$ in $[0,T]$, prove some uniform estimates in $T$,  
and then extend $T$ to infinity. We will only focus on the limit case $T=\infty$ for clarity, and leave the other details for interested readers, but the arguments for general $T$ are essentially same as $T=\infty$.

We now focus on the proof of \eqref{eq: keyapriori}. We will assume the finiteness of $\|M^*(t)\|_{L_{\omega}^{\alpha}L_{t}^{\alpha}(\Omega \times [0,\infty))}$ in the whole article whenever we study the nonlinear problem.
 
For notation convenience, let us denote 
\begin{equation}\label{eq: xs}
 \|f\|_{X^{s_{\alpha},\tilde{\beta}'}}=\|\langle x \rangle^{-100}f\|_{W^{s_{\alpha},\tilde{\beta}'}}.	
\end{equation}
and 
\begin{equation}\label{eq: xxx}
\|f\|_{\XXX}:=\|f\|_{L_{x}^{\beta}}+\|f\|_{X^{s_{\alpha},\tilde{\beta}'}}.
\end{equation}
Via Burkholder inequality, one may estimate by
\begin{equation}
\begin{aligned}
&\|M^{*}(t)\|_{L_{\omega}^{\alpha}L_{t}^{\alpha}(\Omega\times[0,\infty))}\\
=&\|M^{*}\|_{L_{t}^{\alpha}L_{\omega}^{\alpha}(\Omega\times[0,\infty))}\\
\lesssim 
&\eps \|(\int_{0}^{t}\|H(t-s)Vu(s)\|_{X}^{2}ds)^{1/2}\|_{L_{t}^{\alpha}L_{\omega}^{\alpha}(\Omega\times[0,\infty))}\\
=&\eps \|(\int_{0}^{t}\|H(t-s)Vu(s)\|_{X}^{2}ds)^{1/2}\|_{L_{\omega}^{\alpha}L_{t}^{\alpha}(\Omega\times[0,\infty))}.
\end{aligned}
\end{equation}

Note that this $\eps$ is crucial for the close of our estimates, and will be chosen small enough finally.

And \eqref{eq: keyapriori} follows from the following two facts.

One is a purely deterministic functional analysis lemma.
\begin{lemma}\label{lem: pdfa}
One has the following estimate\footnote{As usual, we assume both side of the inequality is finite.},
\begin{equation}\label{eq: pdfa}
\|(\int_{0}^{t}\|H(t-s)VG(s)\|_{\XXX}^{2}ds)^{1/2}\|_{L_{t}^{\alpha}[0,\infty)}\lesssim \|G\|_{L_{t}^{\alpha}\XXX([0,\infty)\times \mathbb{R}^{3})}.
\end{equation}
\end{lemma}

And another one is an enhancement of Prop \ref{prop: nlmain1},
\begin{proposition}\label{prop: nlmain2}
For any $[0,T]$, $T$ finite or equal to $\infty$. One can find a decomposition\footnote{This decomposition may depending on $T$, but it is (of course) crucial that all the implicit constants in the inequality does not depend on $T$.} of $u:=u_{1}+u_{2}$, such that
\begin{equation}\label{eq: controlu1}
\|u_{1}\|_{L_{\omega}^{\alpha}L_{t}^{\alpha}\XXX(\domaint)}\lesssim 1+\|M^{*}(t)\|_{L_{\omega,t}^{\alpha}(\Omega\times [0,T])}
\end{equation}
and, a.s. in $\omega$, 
\begin{equation} \label{eq: controlu2}
\|u_{2}(t)\|_{\XXX}\leq M^{*}(t).	
\end{equation}
\end{proposition}

\begin{remark}
Proposition~\ref{prop: nlmain2} could also be stated more concisely as
\begin{equation}\label{eq: working}
\|u\|_{L_{\omega}^{\alpha}L_{t}^{\alpha}\XXX(\domaint)}\lesssim 1+\|M^{*}(t)\|_{L_{\omega,t}^{\alpha}(\Omega\times [0,T])}.
\end{equation}
But we want to highlight this decomposition and hopefully it will play more role in the future study. And one will see in the proof the first part $u_{1}$ is understood by studying \eqref{eq: hnls} and its perturbation. And $u_{2}$ is estimated by the maximal object, $M^{*}$, itself.
\end{remark}

We will prove Lemma \ref{lem: pdfa} and Prop \ref{prop: nlmain2} in Section \ref{sec: mm}. The proof relies on the material in Section \ref{sec: md}, which deals with properties for damped linear and nonlinear Schr\"odinger equations. We end this section by proving \eqref{eq: keyapriori} assuming Lemma \ref{lem: pdfa} and Prop \ref{prop: nlmain2}.

\subsection{Proof of \eqref{eq: keyapriori} assuming Lemma \ref{lem: pdfa} and Prop \ref{prop: nlmain2}}
We first check that \eqref{eq: controlu1} and \eqref{eq: controlu2} implies \eqref{eq: working}. This follows from 
\begin{equation}
	\|u_{2}\|_{L_{\omega}^{\alpha}L_{t}^{\alpha}\XXX(\domaint)}\lesssim 1+\|M^{*}(t)\|_{L_{\omega,t}^{\alpha}(\Omega\times [0,T])}.
\end{equation}
But this is follows by taking $L_{\omega,t}^{\alpha}$ on both sides of \eqref{eq: controlu2}. Now, plug \eqref{eq: working} into \eqref{eq: pdfa}, and \eqref{eq: keyapriori} follows.

\section{Properties for damped model and modified stability} \label{sec: md}

In this section,  we present the needed properties for damped model. The proof of Theorem \ref{thm: main} only relies on the material in Subsection \ref{subsec: ls}. The proof of Theorem \ref{thm: main0} crucially relies on the material in Subsection \ref{subsec: ms}.   We remark that, except for Lemma \ref{lem: ls2}, the smallness assumptions on $\eps$ is unnecessary, and we conjecture one may remove this smallness assumption for this Lemma. However, the whole article depends on the smallness of $\eps$ anyway, unless one uses some extra idea. 
\subsection{Dispersive estimates and Strichartz estimates for damped model}\label{subsec: ls}
We start with the dispersive estimates and Strichartz estimates.  Let $H$ be the linear propagator to \eqref{eq: dls}.

Based on the important observation in \cite{JSS}, the term $i \eps V^{2}w$ can be treated in a perturbative way (see also \cite{Schlagsurvey} for a nice survey). We also note that for large potential $V$ (without $\eps$), the estimate also holds, see \cite{WXP}.
\begin{lemma}\label{lem: dispersiveh}
One has the same dispersive estimate as $e^{it\Delta}$ in the sense
\begin{equation}\label{eq: dispersiveh}
\|H(t)f\|_{L_{x}^{\infty}}\lesssim t^{-3/2}\|f\|_{L_{x}^{1}},  t>0
\end{equation}
and Riesz-Thorin interpolation gives
\begin{equation}\label{eq:dispersivehp}
\|H(t)f\|_{L_{x}^{p}}\lesssim t^{-3(\frac{1}{2}-\frac{1}{p})}\|f\|_{L_{x}^{p'}}, p\geq 2.
\end{equation}
\end{lemma}

We also have the usual Strichartz estimate.

\begin{lemma}\label{lem: strichartzh}
For all $(q,r)$ admissible, one has 
\begin{equation}\label{eq: strichartzh}
\|H(t)f\|_{L_{t}^{q}L_{x}^{r}([0,\infty)\times \mathbb{R}^{3})}\lesssim \|f\|_{L_{x}^{2}},
\end{equation}
and
\begin{equation}\label{eq: strichartzh2}
\|\int_{a}^{t}H(t-s)F(s)ds\|_{S(a,\infty)}\lesssim \|F\|_{N(a,\infty)},
\end{equation}
where we recall $S$ is the Strichartz norm and $N$ is the dual Strichartz norm.
\end{lemma}
We point out both Lemma \ref{lem: dispersiveh}, Lemma \ref{lem: strichartzh} hold without small assumptions on $\eps$.
\subsection{Local smoothing estimates for damped model}
Next, we present two local smoothing Lemma for the damped model. 
\begin{lemma}\label{lem: ls1}
For $\eps$ small enough, let $H$ be the linear propagator to \eqref{eq: dls}, one has 
\begin{equation}\label{eq: ls1}
\int_{0}^{\infty} \int \langle x \rangle^{-1-\delta}|\nabla^{1/2} H(t)f|^{2}\lesssim_{\delta} \|f\|_{L_{x}^{2}}^{2}.
\end{equation}
\end{lemma}
This follows from the Duhamel Formula, the (dual) local smoothing estimate,  and the (endpoint) Strichartz estimate associated with $H$, and the fact that $V$ is localized in space.
\begin{lemma}\label{lem: ls2}
For $\eps$ small enough, let $H$ be the linear propagator to \eqref{eq: dls}, one has 
\begin{equation}\label{eq: ls2}
\|\nabla H(t)f\|_{L_{x}^{2}(B_{R})}\lesssim \frac{\langle R \rangle^{3/2}}{t}\|\langle x \rangle f\|_{L_{x}^{2}}, \quad t>1,
\end{equation}
\end{lemma}

 Lemma \ref{lem: ls2} will be obtained based the same estimates for linear Schr"odinger equation in a perturbative way.  Lemma \ref{lem: ls2} is probably not optimal but suffice for the purpose of this article.

\subsubsection{Proof of Lemma \ref{lem: ls2}}
With \eqref{eq: sls2}, we only need to prove, (from a standard bootstrap view point and using the fact that $\eps$ is small),
\begin{lemma}\label{lem: miniboot}
Assuming 
\begin{equation}\label{eq: bamini}
\|\nabla H(t)f\|_{L_{x}^{2}(B_{R})}\leq \frac{C\langle R \rangle^{3/2}}{t}\|\langle x \rangle f\|_{L_{x}^{2}},
\end{equation}
then one has, for $\tilde{C}$ only depending on $C$, such that
\begin{equation}
\|\nabla\int e^{i(t-s)\Delta}V^{2}H(s)f\|_{L_{x}^{2}(B_{R})}\leq \tilde{C}\frac{\langle R \rangle^{3/2}}{t}\|\langle x \rangle f\|_{L_{x}^{2}}.
\end{equation}

\end{lemma}
\begin{proof}[Proof of Lemma \ref{lem: miniboot}]
We may only consider $R\geq 1$. Let us fix $R$. We recall $V$ is fast decaying. And we fix $t\geq 1$.

We first note that for any $g$ with
\begin{equation}
\|\nabla g\|_{L_{x}^{2}(B_{R})}\leq C\langle R \rangle^{3/2}B\text{ for some } B>0
\end{equation}
we have, thanks to the fast decay of $V$, that
\begin{equation}
\|Vg\|_{H^{1}}\lesssim CB.	
\end{equation}
We may only consider $t\gg 1$. Indeed, if $t\lesssim 1$, in $[t/2,t]$, we estimate via
\begin{equation}\label{eq: 111}
\bigg\| \int_{t/2}^{t}e^{i(t-s)\Delta}\nabla (V^{2}H(s)f) \bigg\| \lesssim t\sup_{s\in [\frac{t}{2}, t]} \|V^{2}H(s)f\|_{H^{1}_{x}}\;.
\end{equation}
On $s \in [0,t/2]$, since $|t-s|\geq t/2$ and $V^{2}H(s)f$ is uniformly localized in space, we estimate via
\begin{equation}\label{eq: 1111}
\bigg\| \int_{0}^{t/2}e^{i(t-s)\Delta}\nabla (V^{2}H(s)f) \bigg\| \lesssim t \frac{1}{t/2} \|VH(s)f\|_{L_{x}^{2}} \lesssim \|\scal{x} f\|_{L_{x}^{2}}. 
\end{equation}
We now plug in bootstrap assumption \eqref{eq: bamini}, so \eqref{eq: 111} and \eqref{eq: 1111} can be controlled by $t\frac{1}{t}\|\langle x \rangle f\|_{2}$ and the desired estimates follow. (Note that the $R^{3/2}$ is absorbed by fast decay of $V$.)

Now, let us fix $t\gg 1$, we consider three ranges of $s$. When $s\in [t-1,t]$, the estimate is similar to the previous, and we estimate via
\begin{equation}
	\bigg\|\int_{t-1}^{t}e^{i(t-s)\Delta}\nabla (V^{2}H(s)f) ds \bigg\|_{L_{x}^{2}} \lesssim \sup_{s\in [t-1,t]}\|V^{2}H(s)f\|_{H^{1}}\lesssim C\frac{1}{t}\|\langle x \rangle f\|_{2}.
\end{equation}
When $s\in [0,1]$, we simply using $\|\langle x \rangle V^{2}H(s)f\|_{L_{x}^{2}}\lesssim \|f\|_{2}$, and apply \eqref{eq: sls2}. When $s\in [1,t-1]$, we first apply point-wise estimate
\begin{equation}
\|e^{i(t-s)}\nabla (V^{2}H(s))\|_{L_{x}^{\infty}}\lesssim (t-s)^{-3/2}\|\nabla VH(s)\|_{L_{x}^{2}}\lesssim (t-s)^{-3/2}s^{-1}\|\langle x \rangle f\|_{L_{x}^{2}}	
\end{equation}
Now, apply
\begin{equation}
\int_{1}^{t-1}(t-s)^{-3/2}s^{-1} ds \lesssim t^{-1}\;,
\end{equation}
and using the fact locally $L^{\infty}$ embeds $L_{x}^{2}(B_{R})$ with a constant $R^{3/2}$. The desired estimates now follow.
\end{proof}
\subsection{Modified stability}\label{subsec: ms}
The material in this subsection is essentially Lemmas 4.3 and 4.4 and Proposition 4.6 in \cite{fan2020long}, but we have made those results stronger by taking the local smoothing into considerations.\\

As shown\footnote{The setting were slightly different, but the proof works line by line the same if one plugs $\eps_{0}=0$ in Lemma 4.3 or formula (4.19) in \cite{fan2020long}} in Lemma 4.3, \cite{fan2020long}, one can apply Dodson's result, Theorem \ref{thm: dodson} as an black box to establish the desired space time bound for 
\begin{equation}\label{eq: hnls}
\begin{aligned}
iu_{t}+\Delta u=-i\eps^{2}V^{2}u+|u|^{4/3}u,\\
u(0, x)=u_{0}\in L_{x}^{2}.
\end{aligned}
\end{equation}
(Note that all estimates will be independent of $\eps$, as far as $\eps$ is small enough.)

We need a slightly stronger version of Lemma 4.3 in \cite{fan2020long}.

We have
\begin{lemma}\label{lem: hnlsglobal}
Let	$u$ solves \eqref{eq: hnls} with initial data $u_{0}\in L_{x}^{2}$, then $u$ is global forward in time $[0,\infty)$, with estimates
\begin{equation}\label{eq: hnlsg}
\|u\|_{L_{t}^{\alpha}\XXX([0,\infty)\times\mathbb{R}^{3})}\lesssim_{\|u_{0}\|_{L_{x}^{2}}} 1.
\end{equation}
\end{lemma}
We recall $(\alpha, \beta)=(4-,3+)$ is admissible. We recall \eqref{eq: xs}, \eqref{eq: xxx} for the notation $\XXX$.
\begin{proof}[Proof of Lemma \ref{lem: hnlsglobal}]
We only need to prove
\begin{equation}
\|u\|_{L_{t}^{\alpha}\XXX([0,\infty)\times\mathbb{R}^{3})}\lesssim_{\|u_{0}\|_{L_{x}^{2}}} 1.
\end{equation}

And the other part of Lemma \ref{lem: hnlsglobal} is already presented in Lemma \cite{fan2020long}.

Indeed, a direct integration by parts, (just computing the time derivative of $\|u(t)\|_{L_{x}^{2}}$), gives
\begin{equation}
\int_{0}^{t}\int \eps^{2}V^{2}|u|^{2}dxds= \|u_{0}\|^{2}_{L_{x}^{2}}-\|u(t)\|^{2}_{L_{x}^{2}}.	
\end{equation}
Thus, $\eps^{2}V^{2}u$ is in the dual space $L_{t}^{2}L_{x}^{6/5}$, and its size is controlled via $\|u_{0}\|_{L_{x}^{2}}$. 

We also know $|u|^{4/3}u \in L_{t}^{a'}L_{x}^{b'}$ for some $(a,b)$ admissible since $u$ is in $L_{t}^{\alpha}L_{x}^{\beta}$, whose size is also controlled by $\|u_{0}\|_{L_{x}^{2}}$ as proved in Lemma 4.3 in \cite{fan2020long}. Via Duhamel formula\footnote{Here, however, we write down the Duhamel formula based on usual Schr\"odinger operator.}, we have
\begin{equation}
u(t)=e^{it\Delta}u_{0}-i\int_{0}^{t}e^{i(t-s)\Delta}(|u|^{4/3}u+i\eps^{2}V^{2}u)ds.	
\end{equation}
Thus, the desired estimates follow from
\begin{equation}\label{eq: mmm}
\|e^{it\Delta}f\|_{L_{t}^{\alpha}X^{s_{\alpha},\tilde{\beta}'}}\lesssim \|f\|_{L_{x}^{2}},
\end{equation}
and 
\begin{equation}\label{eq: nnn}
\bigg\| \int_{0}^{t}e^{i(t-s)\Delta}F(s)ds \bigg\|_{L_{t}^{\alpha}X^{s_{\alpha},\tilde{\beta}}(\mathbb{R}^{+}\times \mathbb{R}^{3})}\lesssim \|F\|_{L_{t}^{a'}L_{x}^{b'}}, 
\end{equation}
where $(a,b)$ is an admissible pair. The former estimate follows from the local smoothing estimate for free Schr\"odinger, interpolated with mass conservation law, see Remark~\ref{rem: num} for more numerics. The latter follows from the following un-retarded estimate via Christ-Kiselev Lemma (\cite{christ2001maximal}):
\begin{equation}
\Big\|\int_{0}^{\infty}e^{i(t-s)\Delta}F(s)ds \Big\|_{L_{t}^{\alpha}X^{s_{\alpha},\tilde{\beta}}(\mathbb{R}^{+}\times \mathbb{R}^{3})}\lesssim \|F\|_{L_{t}^{a'}L_{x}^{b'}}, \text{ where } a,b \text{ admissible}.
\end{equation}
But this is a consequence of  \eqref{eq: mmm} and Strichartz estimate
\begin{equation}
\Big\|\int_{0}^{\infty}e^{-is\Delta} F(s) ds \Big\|_{L_{x}^{2}}\lesssim \|F\|_{L_{t}^{a'}L_{x}^{b'}}.
\end{equation}
This completes the proof. 
\end{proof}

We are ready to present a more refined version of Proposition 4.6 in \cite{fan2020long}. Note that the main point is that one needs to state the stability in the form of Duhamel Formula.
\begin{proposition}\label{prop: lcmdstability}
Let $u$ solve in some interval $[t_{1}, t_{2}]$ the equation
\begin{equation}
u(t)=H(t-t_{1})u(t_{1})-i\int_{t_{1}}^{t} H(t-s) (|u(s)|^{4/3}u(s)) ds + \eta(t),
\end{equation}
with $\|u(t)\|_{L_{t}^{\infty}L_{x}^{2}}\leq M$, and $\eta(a)=0$, then for $(\alpha,\beta)$ admissible and $\frac{7}{3}\leq \alpha\leq \frac{14}{3}$, there exists $e_{M},B_{M}>0$, such that
If $\|\eta\|_{L_{t}^{\alpha}L_{\beta}}\leq e_{M}$, let 
\begin{equation}
v=u-\eta	
\end{equation}
then 
\begin{equation}\label{eq: vvv}
\|v(t)\|_{L_{t}^{2}L_{x}^{6}\cap L_{t}^{\infty}L_{x}^{2}\cap L_{t}^{\alpha}X^{s_{\alpha},\tilde{\beta}'}([t_{1},t_{2}]\times \mathbb{R}^{3})}\leq \frac{B_{M}}{2},
\end{equation}
As a corollary,
\begin{equation}
\|u\|_{L_{t}^{\alpha}L_{x}^{\beta}([t_{1},t_{2}]\times \mathbb{R}^{3})}\leq B_{M}.
\end{equation}
\begin{remark}
The right endpoint $t_{2}$ can be $\infty$.	
\end{remark}
\end{proposition}

\begin{proof}
The key to the proof, as in the proof of Prop 4.6, \cite{fan2020long}, is to study the equation of $v$ rather than $u$. And the $L_{t}^{2}L_{x}^{6}\cap L_{t}^{\infty}L_{x}^{2}$ bound was already proven there\footnote{Again, one may take the $\eps_{0}=0$ in (4.38), \cite{fan2020long}, and the proof will go through}.

We only need to do the $L_{t}^{\alpha}X^{s_{\alpha},\tilde{\beta}}$.

Note that $v$ solves
\begin{equation}
\begin{cases}
iv_{t}-\Delta v+i\eps^{2}V^{2}v=(|v+\eta|^{4/3})(v+\eta),\\
v(t_{1})=u(t_{1})
\end{cases}
\end{equation}
Let $F_{1}=-i\eps^{2}V^{2}v$, then $F_{1}\in L_{t}^{2}L_{x}^{6/5}$.

Let $F_{2}=(|v+\eta|^{4/3})(v+\eta)$, then $F_{2}$ for some $(a,b)\in L_{t}^{a'}L_{x}^{b'}$ admissible, (since $v$, $\eta$ both in $L_{t}^{\alpha}L_{x}^{\beta}$)
, And by Duhamel Formula, one has 
\begin{equation}
v(t)=e^{i(t-t_{1})\Delta}v(t_{1})-i\int_{t_1}^{t}e^{i(t-s)\Delta}(F_{1}+F_{2})ds.
\end{equation}
The rest part of the proof are same as the proof of \eqref{eq: mmm} and \eqref{eq: nnn}.
\end{proof}

\section{Main estimates for the linear model: proof of Lemmas~\ref{lem: linearesmin} and~\ref{lem: linearesmain}} \label{sec: linearmain}

\subsection{Proof of Lemma \ref{lem: linearesmin}}
We fix $(\alpha,\beta)=(4-,3+)$ an admissible pair. We will split the (stochastic) integral $\int_{0}^{t}$ into $\int_{0}^{t-1\vee 0}$ and $\int_{t-1\vee 0}^{t}$. We record two simple but crucial numeric facts in the following remark.
\begin{remark}\label{rem: num}
\begin{itemize}
\item For the dispersive estimate $\|H(s)f\|_{L_{x}^{\beta}}\lesssim t^{-3(\frac{1}{2}-\frac{1}{\beta})}$, $\beta>3$ implies $-3(\frac{1}{2}-\frac{1}{\beta})<-1/2$.
\item For $f\in L_{x}^{2}$, and $V$ smooth and fast decaying, we know via Lemma \ref{lem: ls1} (and $\|H(t)f\|_{L_{x}^{2}}\leq \|f\|_{L_{x}^{2}}$), that $\|VH(s)f\|_{L_{t}^{2}W_{x}^{\frac{1}{2},p}([0,\infty)\times \mathbb{R}^{3})\cap L_{t}^{\infty}L_{x}^{p}([0,\infty)\times \mathbb{R}^{3})} <\infty$ for all $1<p\leq 2$. Thus, we know it is in $L_{t}^{\alpha}W^{s_{\alpha},q'}$ for all $q\geq 2$, where $s_{\alpha}=\frac{\theta}{2}+0\times{(1-\theta)}$ and $\frac{1}{\alpha}=\frac{\theta}{2}+\frac{1-\theta}{\infty}$, i.e. $s_{\alpha}=\frac{1}{\alpha}$. One may compute the $\tilde{\beta}$ such that $W^{s_{\alpha},\tilde{\beta}}$ (just) embeds $L^{\beta}$, one has $-s_{\alpha}+\frac{3}{\tilde{\beta}}=\frac{3}{\beta}=\frac{3}{2}-\frac{2}{\alpha}$. This give $\frac{3}{\tilde{\beta}}=\frac{3}{2}-\frac{1}{\alpha}$, note that $\alpha>2$ implies $\tilde{\beta}<3$, and note that ${-3(\frac{1}{2}-\frac{1}{\tilde{\beta}})}>-1/2$. We also note $\tilde{\beta}>2$ always hold, thus $VH(s)f$ is in $L_{t}^{\alpha}W^{s_{\alpha},\tilde{\beta}'}$.
\item We summarize the above discussion into estimates
\begin{equation}\label{eq: numls1}
\|e^{i t \Delta}f\|_{L_{t}^{\alpha}X^{s_{\alpha},\tilde{\beta}'}(\mathbb{R}\times \mathbb{R}^{3})}\lesssim \|f\|_{L_{x}^{2}}
\end{equation}
and 
\begin{equation}\label{eq: numls2}
\|e^{it\Delta}f\|_{L_{t}^{\alpha}X^{s_{\alpha},\tilde{\beta}'}([0,\infty)\times \mathbb{R}^{3})}\lesssim \|f\|_{L_{x}^{2}}.
\end{equation}
And \eqref{eq: numls1}, \eqref{eq: numls2} also hold if one replaces $e^{it\Delta}$ by $H(t)$.

\end{itemize}
\end{remark}
We carry on the proof.  We will fix admissible pair $(\alpha,\beta)=(4-,3+)$ but $(2+,6-)$.

We first prove following two estimates regarding linear solutions.
\begin{equation}\label{eq: estlinearsmoothing}
	\|VH(t)u_{0}\|_{L_{t}^{\alpha}W^{s_{\alpha}\tilde{\beta}'}([0,\infty)\times \mathbb{R}^{3})}\lesssim \|u_{0}\|_{L_{x}^{2}}
\end{equation}
and
\begin{equation}\label{eq: estlinearsti}
\|VH(t)u_{0}\|_{L_{t}^{\alpha}L_{x}^{\beta'}([0,\infty)\times \mathbb{R}^{3})}\lesssim \|u_{0}\|_{L_{x}^{2}}.
\end{equation}

We will let 
\begin{equation}
F(t):=H(t)u_{0},	
\end{equation}
in the following and we only need to estimate 
$\int_{t-1\vee 0}^{t}H(t-s)VF(s)dB_{s}$ and $\int_{0}^{t-1\vee 0}F(s)dB_{s}$.

\begin{proof}[Proof of estimates \eqref{eq: estlinearsmoothing} and \eqref{eq: estlinearsti}]

First observe \eqref{eq: estlinearsti} is a direct consequence of Strichartz estimate, Lemma \ref{lem: strichartzh} and the fact that $V$ is well localized.

We now turn to the proof of \eqref{eq: estlinearsmoothing}. But it is a direct consequence of \eqref{eq: ls1} and our choice of $s_{\alpha}$, and the fact that $V$ is well localized and smoothing, as explained in Remark \ref{rem: num}.
\end{proof}

We now turn to the estimate for $\int_{t-1\vee 0}^{t}H(t-s)F(s)dB_{s}$. We apply dispersive estimate and Sobolev embedding to obtain
\begin{equation}\label{eq: disso}
\begin{aligned}
\|H(t-s)VF(s)\|_{L_{x}^{\beta}} &\lesssim \|H(t-s)F(s)\|_{W_{x}^{s,\tilde{\beta}}}\\
&\lesssim (t-s)^{-3(\frac{1}{2}-\frac{1}{\tilde{\beta}})}\|F(s)\|_{W_{x}^{s,\tilde{\beta}'}}.
\end{aligned}
\end{equation}
Next, for all $t$ fixed, by Burkholder inequality, we have
\begin{equation}\label{eq: minib1}
\bigg\| \int_{t-1\vee 0}^{t}H(t-s)F(s)dB_{s} \bigg\|_{L_{\omega}^{\alpha}W_{x}^{s_{\alpha},\tilde{\beta}}} \lesssim \bigg\|\int_{t-1\vee 0}^{t}\|H(t-s)F(s)\|^{2}_{W_{x}^{s_{\alpha},\tilde{\beta}}} ds \bigg\|_{L_{\omega}^{\alpha/2}}^{1/2}.
\end{equation}
Thus, summarizing \eqref{eq: disso}, \eqref{eq: minib1}, we obtain\footnote{We choose space of $L_{\omega,t}^{\alpha} $ type so that we can switch the order of integration.}
\begin{equation}\label{eq: minifinal1}
\begin{aligned}
&\phantom{111}\bigg\|\int_{t-1\vee 0}^{t}H(t-s) (V H(s)u_{0}) ds \bigg\|_{L_{\omega}^{\alpha}L_{t}^{\alpha}L_{x}^{\beta}(\domain)}^{\alpha}\\
&\lesssim \int_{\R^{+}} \left(\int_{t-1\vee 0}^{t}\|H(t-s) (VH(s)u_{0}) \|^{2}_{W_{x}^{s_{\alpha},\tilde{\beta}}} ds\right)^{\alpha/2} dt\\
&\lesssim \bigg\|\int_{t-1\vee 0}^{t}(t-s)^{-6(\frac{1}{2}-\frac{1}{\tilde{\beta}})}\|VH(s)u_{0}\|_{W_{x}^{s,\tilde{\beta}'}}^{2} ds \bigg\|_{L_{t}^{\alpha/2}}.
\end{aligned}
\end{equation}
An application of Young inequality closes the estimate via \eqref{eq: estlinearsmoothing} since 
$(t-s)^{-6(\frac{1}{2}-\frac{1}{\tilde{\beta}})}$ is locally integrable (as a function of $s$, localized around $0\leq t-s\leq 1$). We refer to Remark~\ref{rem: num} for more on the numerics of the exponents.

Now, we estimate the part $\int_{0}^{t-1\vee 0}H(t-s)F(s)dB_{s}$. We may only consider the estimate for the part $t\gg 1$, otherwise we just repeat the previous argument for the first part. Recall $\beta>3$, so that $3(\frac{1}{2}-\frac{1}{\beta})>1/2$. With Burkholder inequality (recall we may only work on $t\gg 1$ and for example $t\geq 2$), we have
\begin{equation}
\begin{aligned}
&\phantom{111} \bigg\|\int_{0}^{t-1}H(t-s)F(s)dB_{s} \bigg\|_{L_{\omega}^{\alpha}L_{t}^{\alpha}L_{x}^{\beta}(\Omega\times [2,\infty)\times \mathbb{R}^{3})}^{\alpha}\\
&\lesssim \bigg\|\int_{0}^{t-1}(t-s)^{6(\frac{1}{2}-\frac{1}{{\beta}})}\|F(s)\|_{L_{t}^{\alpha}L_{x}^{\beta'}}^{2} \bigg\|_{L_{\omega}^{\alpha/2}}.
\end{aligned}
\end{equation}
Using \eqref{eq: estlinearsti} and the fact $(t-s)^{6(\frac{1}{2}-\frac{1}{{\beta}})}$ is integrable at $|t-s|\geq 1$ and the desired estimate follows by Young inequality.

\subsection{Proof of Lemma \ref{lem: linearesmain}}
Now we turn to the proof Lemma~\ref{lem: linearesmain}. We first observe that \eqref{eq: linearlocalpart} and \eqref{eq: linearglobalpart} follow from \eqref{eq: lineara1} and \eqref{eq: lineara2} via arguments similar to the previous section. Indeed, now we are estimating $\int H(t-s)A(s)dB_{s}$ while in the previous section we are estimating $\int H(t-s)F(s) ds$, and we just need to replace \eqref{eq: estlinearsmoothing} and \eqref{eq: estlinearsti} by \eqref{eq: lineara1} and \eqref{eq: lineara2}. Since the arguments are almost line by line same in other parts, we omit the details here. 

We now turn to the proof of \eqref{eq: lineara1} and \eqref{eq: lineara2}. We remark we choose the space of $L_{\omega,t}^{\alpha}$ type so we can switch the order of integral in $t$ and $\omega$ freely. 

\begin{proof}[Proof of \eqref{eq: lineara1}]
By Burkholder inequality, for $t$ fixed, one has 
\begin{equation}
\|A(t)\|_{L_{\omega}^{\alpha}W_{x}^{s_{\alpha},\tilde{\beta}'}} \lesssim  \bigg\| \int_{0}^{t}\|V(x)H(t-s)V(x)u\|_{W_{x}^{s_{\alpha},\tilde{\beta}'}}^{2} ds \bigg\|^{1/2}_{L_{\omega}^{\alpha/2}}.		
\end{equation}
We are left with the proof of
\begin{equation}\label{eq: lsmains}
\bigg\|\int_{0}^{t}\|V(x)H(t-s)V(x)u(s)\|_{W_{x}^{s_{\alpha},\tilde{\beta}'}}^{2}ds \bigg\|_{{L_{\omega}^{\alpha/2}L_{t}^{\alpha/2}}} \lesssim \|u\|_{L_{\omega}^{\alpha}L_{t}^{\alpha}L_{x}^{\beta}}^{2}.	
\end{equation}
We note that one should not worry too much about the integrability in $x$ but focus rather on regularity in $x$, since $V$ is Schwarz.  And one can also process in a deterministic way, estimate via fixing $\omega$ and take $L_{\omega}^{\alpha/2}$ at the end.

One may split the integral in \eqref{eq: lsmains} into $|t-s|\leq 1$ and $|t-s|\geq 1$. For $|t-s|\leq 1$, we first apply H\"older inequality, relying\footnote{Note that $\alpha$ is always no smaller than $2$, as far as $(\alpha,\beta)$ is admissible.} on $\alpha\geq 2$ that
\begin{equation}
\begin{aligned}
&\phantom{111}\bigg\| \int_{0}^{t}\chi_{|t-s|\leq 1}\|VH(t-s)Vu(s)\|_{W_{x}^{s_{\alpha},\beta'}}^{2} \bigg\|_{L_{t}^{\alpha/2}}^{\alpha/2}\\
&= \bigg\|\left(\int_{0}^{t}\chi_{|t-s|\leq 1}\|VH(t-s)Vu(s)\|_{W_{x}^{s_{\alpha},\beta'}}^{2} ds \right)^{1/2}  \bigg\|_{L_{t}^{\alpha}}^{\alpha}\\
&\lesssim \int_{0}^{\infty}\int_{0}^{t}\|VH(t-s)Vu(s)\|_{W_{x}^{s_{\alpha},\beta'}}^{\alpha} ds dt.
\end{aligned}	
\end{equation}
Switching the order of integral, applying $\|VH(t)f\|_{L_{t}^{\alpha}W^{s_{\alpha},p}}\lesssim \|f\|_{L_{x}^{2}}$, $p\leq 2$, as explained in Remark~\ref{rem: num}, one obtains
\begin{equation}
\int_{0}^{\infty}\int_{0}^{t}\|VH(t-s)Vu(s)\|_{W_{x}^{s_{\alpha},\beta'}}^{\alpha} ds dt \lesssim \int_{\R^+} \|Vu(s)\|_{L_{x}^{2}}^{\alpha} ds \lesssim \|u\|_{L_{t}^{\alpha}L_{x}^{\beta}}^{\alpha}.
\end{equation}
Integrating in $\omega$ variable and the desired estimate will follow.

For $|t-s|\geq 1$, applying the local smoothing estimate in Lemma~\ref{lem: ls2} and using localization of $V$, we get
\begin{equation}
\|\langle x \rangle^{-10}H(t-s)Vu(s)\|_{\dot{H}^{1}}\lesssim 	\|u(s)\|_{L_{x}^{\beta}}(t-s)^{-2}.
\end{equation}
Thus, we derive
\begin{equation}
\|V(x)H(t-s)Vu(s)\|_{W^{1, \tilde{\beta}'}}\lesssim \|u(s)\|_{L_{x}^{\beta}}(t-s)^{-1}
\end{equation}
By Young's inequality, we obtain
\begin{equation}
\bigg \|\int_{0}^{t}\chi_{|t-s|\geq 1}	 \|u(s)\|^{2}_{L_{x}^{\beta}}(t-s)^{-2} \bigg\|_{L_{t}^{\alpha/2}}^{1/2}\lesssim \|u\|_{L_{t}^{\alpha}L_{x}^{\beta}}
\end{equation}
Taking $L_{\omega}^{\alpha}$ on both sides and desired estimate will follow.
\end{proof}

\begin{proof}[Proof of \eqref{eq: lineara2}]
By Burkholder inequality, for $t$ fixed, one has 
\begin{equation}
\|A(t)\|_{L_{\omega}^{\alpha}L_{x}^{\beta'}} \lesssim \bigg\|\int_{0}^{t}V(x)H(t-s)V(x)u\|_{L_{x}^{\beta'}}^{2} ds \bigg\|^{1/2}_{L_{\omega}^{\alpha/2}}\;,
\end{equation}
and we only need to estimate  $\|\int_{0}^{t}V(x)H(t-s)V(x)u\|_{L_{x}^{\beta}}^{2} ds\|^{1/2}_{L_{\omega,t}^{\alpha/2}}$.

Using localization of $V$, we have
\begin{equation}
\|VH(t-s)V(x)u(s)\|_{L_{x}^{\beta'}}\lesssim \|Vu\|_{L_{x}^{2}}\lesssim \|u(s)\|_{L_{x}^{\beta}}, \text{ for } t-s\leq 1	
\end{equation}
and (by dispersive estimate)
\begin{equation}
\|VH(t-s)V(x)u_{s}\|_{L_{x}^{\beta'}}\lesssim \|H(t-s)V(x)u(s)\|_{L_{x}^{\infty}}\lesssim (t-s)^{-3/2}\|u(s)\|_{L_{x}^{\beta}}.	
\end{equation}
To summarize, we have
\begin{equation}
\int_{0}^{t}\|VH(t-s)Vu(s)ds\|_{L_{x}^{\beta'}}^{2}\lesssim \int_{0}^{t}(1+|t-s|)^{-3}\|u(s)\|_{L_{x}^{\beta}}^{2}ds
\end{equation}
Applying Young's inequality in $t$ and taking the $L_{\omega}^{\alpha}$ norm, the desired estimates will follow.
\end{proof}

\section{Main estimates for the nonlinear model} \label{sec: mm}

We first present the proof of Lemma \ref{lem: pdfa}. We recall \eqref{eq: xs}, \eqref{eq: xxx} for notations.
\subsection{Proof of Lemma \ref{lem: pdfa}}
\begin{proof}
In view of splitting the time interval, it suffices to consider two cases:
\begin{equation}
 \bigg\| (\int_0^{t-1} \|H(t-s)VG(s)\|_{\mathbb{X}}^2 ds)^{\frac{1}{2}} \bigg\|_{L_t^{\alpha}[0,\infty)}
\end{equation}
and 
\begin{equation}
\bigg\|(\int_{t-1}^{t} \|H(t-s) VG(s) \|_{\mathbb{X}}^2 ds)^{\frac{1}{2}} \bigg\|_{L_t^{\alpha}[0,\infty)}.
\end{equation}
Also, we note that estimate \eqref{eq: pdfa} is indeed two estimates noticing the definition of the $\mathbb{X}$-norm. We consider the $X^{s_{\alpha},\tilde{\beta}'}$-norm case and the $L_x^{\beta}$-norm case respectively. Thus totally there are four sub-cases.

We consider the left hand side of \eqref{eq: pdfa} takes $X^{s_{\alpha},\tilde{\beta}'}$-norm. For the first case, we use the first local smoothing estimate Lemma \ref{lem: ls1} together with the observation that on finite interval $L^2$-norm can be improved to $L^{\alpha}$-norm ($\alpha>2$)
\begin{align*}
\bigg\|(\int_{t-1}^{t} \|H(t-s)VG(s)\|_{X^{s_{\alpha},\tilde{\beta}'}}^2 ds)^{\frac{1}{2}} \bigg\|_{L_t^{\alpha}[0,\infty)} &\lesssim  \bigg\| \Big( \int_{t-1}^{t} \|H(t-s)VG(s)\|_{X^{s_{\alpha},\tilde{\beta}'}}^\alpha ds \Big)^{\frac{1}{\alpha}} \bigg\|_{L_t^{\alpha}[0,\infty)}\\
&\lesssim   \big\| \|H(t-s)VG(s)\|_{X^{s_{\alpha},\tilde{\beta}'}} \big\|_{L_t^{\alpha}L_s^{\alpha}} \\
&\lesssim   \big\| \|H(t-s)VG(s)\|_{X^{s_{\alpha},\tilde{\beta}'}} \big\|_{L_s^{\alpha}L_t^{\alpha}} \\
&\lesssim  \|G\|_{L_t^{\alpha}\mathbb{X}}. 
\end{align*}
For the second case, we use the second local smoothing estimate \eqref{lem: ls2} together with the Young's inequality,
\begin{align*}
\bigg\| \Big( \int_{0}^{t-1} \|H(t-s)VG(s)\|_{X^{s_{\alpha},\tilde{\beta}'}}^2 ds \Big)^{\frac{1}{2}} \bigg\|_{L_t^{\alpha}[0,\infty)} &\lesssim  \bigg\| \Big(\int_{0}^{t-1} (t-s)^{-2}\|VG(s)\|_{L^{2}_x} ds \Big)^{\frac{1}{2}} \bigg\|_{L_t^{\alpha}[0,\infty)}\\
&\lesssim  \bigg\| \Big(\int_{0}^{t-1} (t-s)^{-2}\|G(s)\|_{L^{\beta}_x}^2 ds \Big)^{\frac{1}{2}} \bigg\|_{L_t^{\alpha}[0,\infty)}\\
&\lesssim  \|G\|_{L_t^{\alpha}\mathbb{X}}. 
\end{align*}
We now consider the left hand side of \eqref{eq: pdfa} takes $L_x^{\beta}$-norm. For the first case, we use the dispersive estimate \eqref{eq: dispersiveh}, Sobolev embedding and the Young's inequality,
\begin{align*}
\bigg\| \Big(\int_{t-1}^{t} \|H(t-s)VG(s)\|_{L_x^{\beta}}^2 ds \Big)^{\frac{1}{2}} \bigg\|_{L_t^{\alpha}[0,\infty)}
&\lesssim  \bigg\| \Big( \int_{t-1}^{t} \|H(t-s)VG(s)\|_{W^{s_{\alpha},\tilde{\beta}}}^2 ds \Big)^{\frac{1}{2}} \bigg\|_{L_t^{\alpha}[0,\infty)}\\
&\lesssim  \bigg\| \Big(\int_{t-1}^{t} (t-s)^{-2(\frac{3}{2}-\frac{3}{\tilde{\beta}})}\|VG(s)\|^{2}_{W^{s_{\alpha},\tilde{\beta}'}} ds \Big)^{\frac{1}{2}} \bigg\|_{L_t^{\alpha}[0,\infty)} \\
&\lesssim  \bigg\| \Big( \int_{t-1}^{t} (t-s)^{-2(\frac{3}{2}-\frac{3}{\beta})} \|G(s)\|_{\mathbb{X}}^2 ds \Big)^{\frac{1}{2}} \bigg\|_{L_t^{\alpha}[0,\infty)} \\
&\lesssim \|G\|_{L_t^{\alpha}\mathbb{X}}. 
\end{align*}
For the second case, we use the dispersive estimate \eqref{eq: dispersiveh} together with the Young's inequality,
\begin{align*}
\bigg\| \Big( \int_{0}^{t-1} \|H(t-s)VG(s)\|_{L_x^{\beta}}^2 ds \Big)^{\frac{1}{2}} \bigg\|_{L_t^{\alpha}[0,\infty)} &\lesssim  \bigg\| \Big( \int_{0}^{t-1} (t-s)^{-2(\frac{3}{2}-\frac{3}{\beta})}\|VG(s)\|_{L^{\beta^{'}}_x} ds \Big)^{\frac{1}{2}} \bigg\|_{L_t^{\alpha}[0,\infty)}\\
&\lesssim \bigg\| \Big( \int_{0}^{t-1} (t-s)^{-2(\frac{3}{2}-\frac{3}{\beta})}\|G(s)\|_{L^{\beta}_x}^2 ds \Big)^{\frac{1}{2}} \bigg\|_{L_t^{\alpha}[0,\infty)}\\
&\lesssim \|G\|_{L_t^{\alpha}\mathbb{X}}. 
\end{align*}
\end{proof}

\subsection{Proof of Proposition \ref{prop: nlmain2}}

The proof of Proposition \ref{prop: nlmain2} is schematically similar to the proofs of Lemmas 3.6 andd 4.10 in \cite{fan2020long}, with extra focus on the local smoothing effect, and in particular relies on the Proposition~\ref{prop: lcmdstability} in the current article, which enhances Proposition 4.6 in \cite{fan2020long}. 

\subsubsection{Step: 1} We fix $(\alpha,\beta)=(4-,3+)$, $M=\|u\|_{L_{\omega}^{\infty}L_{t}^{\infty}L_{x}^{2}}$. Fix $e_{m}$ as in Prop \ref{prop: lcmdstability}. Note that
\begin{equation}\label{eq: as}
	\sum_{A}A^{\alpha}\mathbb{P}(\omega: \|M^{*}\|_{L_{t}^{\alpha}}\sim A)\lesssim \|M^{*}\|_{L_{\omega}^{\alpha}L_{t}^{\alpha}}^{\alpha}+1
\end{equation}
where $A$ ranges over all dyadic integers.

Now we fix an $\omega$ such that $\|M^{*}\|_{L_{t}^{\alpha}([0,\infty))}\sim A$.

\subsubsection{Step: 2} We proceed in a deterministic way. We may divide $[0,\infty)$ into $J$ disjoint interval $\cup_{j}[t_{j},t_{j+1}]$, $(t_{J+1}=\infty)$, so that 
\begin{equation}\label{eq: zzz}
J\lesssim \frac{A^{\alpha}}{\eps_{M}^{\alpha}}+1, \text{ and } \|M^{*}(t)\|_{L_{t}^{\alpha}([t_{j}, t_{j+1}])}\leq e_{M}.	
\end{equation}
Write down the Duhamel formula for $u$ in each $[t_{j},t_{j+1}]$, 
\begin{equation}\label{eq: ggg}
\begin{aligned}
u(t)=H(t-t_{j})u(t_{j})+i\int_{t_{j}}^{t}H(t-s) (|u(s)|^{4/3}u(s)) ds+\eta(t),\\
\eta(t):=\eta(t_{j},t)=i\int_{t_{j}}^{t}H(t-s) (\eps V u(s)) dB_{s}.
\end{aligned}
\end{equation}
Note that $\|\eta(t)\|_{\XXX}\leq M^*(t)$ pointwise in $t$. Now we define, on each $[t_{j},t_{j+1}]$
\begin{equation}
u_{1}=u-\eta, 
u_{2}=\eta	.
\end{equation}
Note that $u_{2}$ already satisfy the property required in Prop \ref{prop: nlmain2}.

\subsubsection{Step 3:} 

We now check the property of $u_{1}$. Applying Proposition~\ref{prop: lcmdstability} on each $[t_{j}, t_{j+1}]$, we have
\begin{equation}
\|u_{1}\|_{L_{t}^{\alpha}\XXX}\leq B_{M}.	
\end{equation}
Thus, by integration in $L_{t}^{\alpha}$ on all those $J$ intervals, we have
\begin{equation}
\|u_{1}\|_{L_{t}^{\alpha}\XXX([0,\infty)\times \mathbb{R}^{3})}\lesssim J	.
\end{equation}
The desired estimates for $u_{1}$ now follows from \eqref{eq: zzz} and \eqref{eq: as}.
\section{On scattering behavior}\label{sec: scattering}
We note that the material in this section can be understood as a refinement of the associated part in \cite{fan2020long}, using the idea in the proof of our main estimate \eqref{eq: estimatemain_nonlinear}.

Recall Duhamel formula
\begin{equation}\label{eq: finalduhamel}
u(t)=H(t)u_{0}-i\int_{0}^{t}H(t-s)Vu(s)dB_{s}-i\int_{0}^{t}H(t-s)|u|^{4/3}uds.	
\end{equation}

and we have already established \eqref{eq: estimatemain_nonlinear}.

Following the arguments in Section 6 of \cite{fan2020long}, the desired scattering behavior will follow if we can prove
\begin{equation}\label{eq: basescattering}
\|u\|_{L_{\omega}^{2}L_{t}^{2}W(\Omega\times [0,\infty)\times \mathbb{R}^{3})}\lesssim 1
\end{equation}
Where
\begin{equation}
\|f\|_{W}:=\|<x>^{-100}f\|_{L_{x}^{2}}.	
\end{equation}
Note that $\|f\|_{W}\lesssim \|f\|_{L_{x}^{6}}.$

Let $\mathbb{Z}$ be defined as 
\begin{equation}
\|f\|_{\mathbb{Z}}:=\|\langle x \rangle^{-10}f\|_{H^{1/2}}
\end{equation}
And we assume $M:=\|u_{0}\|_{L_{\omega}^{\infty}L_{x}^{2}}$ is fixed and $\epsilon$ has been chosen small enough (and fixed).

Via \eqref{eq: finalduhamel}, crucial space time bound \eqref{eq: estimatemain_nonlinear},  we can obtain \eqref{eq: basescattering} if we can prove
\begin{equation}\label{eq: firscattering}
\|i\int_{0}^{t}H(t-s)\epsilon VudB_{s}\|_{L_{\omega}^{2}L_{t}^{2}\mathbb{Z}}\lesssim 1
\end{equation}

Before we carry on the proof, we point out that we will need a stronger bound than \eqref{eq: estimatemain_nonlinear}, i.e. \eqref{eq: keysum} regarding $M^{*}(t)$, which is also established in the previous sections.

We introduce, 
\begin{equation}
M^{*}_{1}(t):=\sup_{a \leq b\leq t}\|\int_{a}^{b}H(t-s)\epsilon Vu dB_{s}\|_{\mathbb{Z}}	
\end{equation}

And clearly \eqref{eq: firscattering} follows from
\begin{equation}\label{eq: secondscattering}
\|M^{*}_{1}(t)\|_{L_{\omega}^{2}L_{t}^{2}(\Omega\times [0,\infty))}\lesssim 1.
\end{equation}

Standard continuity argument or bootstrap technique will yield 
\eqref{eq: secondscattering} once we have \textbf{a priori} estimate, (and $\epsilon$ is small enough),
\begin{equation}\label{eq: m1apriori}
\|M_{1}^{*}(t)\|_{L_{\omega}^{2}L_{t}^{2}(\Omega\times [0,\infty))}\lesssim 1+\textcolor{red}{\epsilon}\|M_{1}^{*}(t)\|_{L_{\omega}^{2}L_{t}^{2}(\Omega\times [0,\infty))}.
\end{equation}

Arguing similarly as the proof of Lemma 3.6 in \cite{fan2020long}, we have the its analogous version\footnote{Basically, $u_{1}$ satisfy all the good properties of linear free Schr"odinger solution, and recall $\|e^{it\Delta}f\|_{L_{t}^{2}\mathbb{Z}}\lesssim \|f\|_{L_{x}^{2}}$} as follows, (with our Prop \ref{prop: lcmdstability}),
\begin{lemma}\label{lem: secondsplit}
Let $u$ be the solution to \eqref{eq: snls} with \eqref{eq: keysum} and \eqref{eq: secondscattering}, then we can split $u=u_{1}+u_{2}$, with
\begin{equation}\label{eq: f1}
\|u_{1}\|_{L_{\omega}^{2}L_{t}^{2}\mathbb{Z}(\Omega\times [0,\infty)\times \mathbb{R}^{3})}\lesssim 1+\|M^{*}\|_{L_{\omega}^{\alpha}L_{t}^{\alpha}}^{\frac{\alpha}{2}}	
\end{equation}
and 
\begin{equation}\label{eq: f2}
\|u_{2}\|_{L_{\omega}^{2}L_{t}^{2}\mathbb{Z}(\Omega\times \times [0,\infty)\times\mathbb{R}^{3})}\lesssim \|M_{1}^{*}(t)\|_{L_{\omega}^{2}L_{t}^{2}(\Omega\times [0,\infty))}.
\end{equation}
\end{lemma}

Applying Lemma \ref{lem: secondsplit},
we further compute, for all $t$ fixed, (and as usual, we may only consider $t\geq 1$), via Burkholder inequality
\begin{equation}
\begin{aligned}
\|M^{*}_{1}(t)\|_{L_{\omega}^{2}}\lesssim \int_{0}^{t}\|\langle x \rangle^{-10}<D>^{1/2}H(t-s)\epsilon Vu(s)ds\|_{L_{\omega}^{2}L_{x}^{2}}^{2}
\end{aligned}
\end{equation}
and note that 
\begin{equation}
\begin{aligned}
 \int_{0}^{t}\|\langle x \rangle^{-10}<D>^{1/2}H(t-s)\epsilon Vu(s)ds\|_{L_{x}^{2}}^{2}\lesssim& \int_{0}^{t}\|\langle x \rangle^{-10}H(t-s)\epsilon V<D>^{1/2} u_{1}\|_{L_{x}^{2}}^{2}ds\\+&\textcolor{red}{\epsilon^{2}}\int_{0}^{t}\|\langle x \rangle^{-10}<D>^{1/2}H(t-s)Vu_{2}\|_{L_{x}^{2}}^{2}ds.
 \end{aligned}
\end{equation}
Clearly, one has for the part with $u_{i}$, $i=1,2$
\begin{equation}
\int_{0}^{t}\|H(t-s) Vu_{i}(s)\|_{\mathbb{Z}}^{2}ds\lesssim \int_{|t-s|<1, s\in [0,t]}\|<D>^{1/2}Vu_{i}(s)\|_{L_{x}^{2}}^{2}+\int_{|t-s|>1, s\in [0,t]}	(t-s)^{-3}\|<D>^{1/2}Vu_{i}(s)\|_{L_{x}^{1}}^{2}.
\end{equation}
Taking $L_{\omega}^{2}L_{x}^{2}$, and it is controlled via $\|u_{1}\|_{L_{\omega}^{2}L_{t}^{2}\mathbb{Z}(\Omega\times [0,\infty)\times \mathbb{R}^{3})}$ and $\|u_{2}\|_{L_{\omega}^{2}L_{t}^{2}\mathbb{Z}(\Omega\times \times [0,\infty)\times\mathbb{R}^{3})}$ thanks to the localization of $V$. Recall \eqref{eq: f1}, \eqref{eq: f2}, we are done.

\section{Covering all non-end point case, proof for \eqref{eq: estimatemain_nonlinear222}}
We need to push $\alpha$ below $7/3$. We do a proof for all $2<\alpha<4$. Let $(\alpha, \beta)$ be admissible, then $\beta>3$.

One can see, in some sense, the computation, at least the algebraic part, goes back to the linear model.

First observe via Lemma \ref{lem: secondsplit} and estimate \eqref{eq: secondscattering}, we obtain
\begin{equation}\label{eq: hhh}
\|<x>^{-100}<\nabla>^{1/2}u\|_{L_{\omega}^{2}L_{t}^{2}L_{x}^{2}}\lesssim 1.	
\end{equation}

And, now, we rewrite the Duhamel Formula of \eqref{eq: snls} as 
\begin{equation}
	u(t)=S(t)u_{0}-i\int_{0}^{t}S(t-s)|u|^{4/3}uds-i\int_{0}^{t}S(t-s)\epsilon Vu(s)dB_{s}-\frac{1}{2}\int S(t-s)\epsilon^{2}V^{2}u(s)ds.
\end{equation}

Except for the stochastic integral term, all other terms can be directly handle by Strichartz estiamte and bounded in $L_{\omega}^{2}L_{t}^{\alpha}L_{x}^{\beta}$.

We only need to handle the stochastic term. We will indeed control it in $L_{\omega}^{\alpha}L_{t}^{\alpha}L_{x}^{\beta}$.

Note that \eqref{eq: hhh} implies, (via interpolation with mass conservation law), see also the algebra in Remark \ref{rem: num}, that 
\begin{equation}
\|Vu\|_{L_{\omega}^{\alpha}L_{t}^{\alpha}W^{s_{\alpha}},\tilde{\beta}'}+\|Vu\|_{L_{\omega}^{\alpha}L_{t}^{\alpha}L_{x}^{\beta'}}\lesssim 1	
\end{equation}
where we let $\dot{W}^{s_{\alpha},\tilde{\beta}}$ embeds $L^{\beta}$, noting that $\tilde{\beta}<3$.

We split the stochastic integral into $|t-s|<1$ and $|t-s|>1$.

For $|t-s|>1$, just apply Burkholder, and observe 
\begin{equation}
\int_{0}^{t}\chi_{t-s>1}\|S(t-s)Vu(s)\|_{L_{x}^{\beta}}^{2}\lesssim \int_{0}^{t}<t-s>^{-6(1/2-1/\beta)}\|Vu(s)\|_{L_{x}^{\tilde{\beta}'}}^{2}
\end{equation}
and an application of Young inequality will end the proof, noting that $6(1/2-1/\beta)>1$.

For $|t-s|<1$,  just apply Burkholder, and observe 
\begin{equation}
\int_{0}^{t}\chi_{t-s<1}\|S(t-s)Vu(s)\|_{L_{x}^{\beta}}^{2}\lesssim \int_{0}^{t}\chi_{t-s<1}(t-s)^{-6(\frac{1}{2}-\frac{1}{\tilde{\beta}})}\|<D>^{s_{\alpha}}Vu(s)\|_{L_{x}^{\tilde{\beta}'}}^{2}
\end{equation}
and 
$6(\frac{1}{2}-\frac{1}{\tilde{\beta}})<1$ since $\tilde{\beta}<3$, and we use Young to conclude.

\bibliographystyle{amsplain}
\bibliographystyle{plain}
\bibliography{BG}

\providecommand{\bysame}{\leavevmode\hbox to3em{\hrulefill}\thinspace}
\providecommand{\MR}{\relax\ifhmode\unskip\space\fi MR }
\providecommand{\MRhref}[2]{%
  \href{http://www.ams.org/mathscinet-getitem?mr=#1}{#2}
}
\providecommand{\href}[2]{#2}
\begin{thebibliography}{10}

\bibitem{BRZ14}
Viorel Barbu, Michael R\"{o}ckner, and Deng Zhang, \emph{Stochastic nonlinear
  {S}chr\"{o}dinger equations with linear multiplicative noise: rescaling
  approach}, J. Nonlinear Sci. \textbf{24} (2014), no.~3, 383--409.
  \MR{3215081}

\bibitem{B}
Zdzis\l~aw Brze\'{z}niak, \emph{On stochastic convolution in {B}anach spaces
  and applications}, Stochastics Stochastics Rep. \textbf{61} (1997), no.~3-4,
  245--295. \MR{1488138}

\bibitem{BP}
Zdzis\l~aw Brze\'{z}niak and Szymon Peszat, \emph{Space-time continuous
  solutions to {SPDE}'s driven by a homogeneous {W}iener process}, Studia Math.
  \textbf{137} (1999), no.~3, 261--299. \MR{1736012}

\bibitem{Brzezniak-Liu-Zhu}
Zdzislaw Brze\'{z}niak, Wei Liu, and Jiahui Zhu, \emph{The stochastic
  {S}trichartz estimates and stochastic nonlinear {S}chr\"{o}dinger equations
  driven by {L}\'{e}vy noise}, Journal of Functional Analysis \textbf{281}
  (2021), no.~1.

\bibitem{Brzezniak-Millet}
Zdzislaw Brze\'{z}niak and Annie Millet, \emph{On the stochastic {S}trichartz
  estimates and the stochastic nonlinear {S}chr\"odinger equation on a compact
  {R}iemannian manifold}, Potential Analysis \textbf{41} (2014), 269--315.

\bibitem{Bur73}
D.~L. Burkholder, \emph{Distribution function inequalities for martingales},
  Ann. Probability \textbf{1} (1973), 19--42. \MR{365692}

\bibitem{BDG73}
D.~L. Burkholder, B.~J. Davis, and R.~F. Gundy, \emph{Integral inequalities for
  convex functions of operators on martingales}, Proceedings of the {S}ixth
  {B}erkeley {S}ymposium on {M}athematical {S}tatistics and {P}robability
  ({U}niv. {C}alifornia, {B}erkeley, {C}alif., 1970/1971), {V}ol. {II}:
  {P}robability theory, 1972, pp.~223--240. \MR{0400380}

\bibitem{cazenave2003semilinear}
Thierry Cazenave, \emph{Semilinear {S}chr{\"o}dinger equations}, vol.~10,
  American Mathematical Soc., 2003.

\bibitem{christ2001maximal}
Michael Christ and Alexander Kiselev, \emph{Maximal functions associated to
  filtrations}, Journal of Functional Analysis \textbf{179} (2001), no.~2,
  409--425.

\bibitem{constantin1987effets}
Peter Constantin and J-C Saut, \emph{Effets r{\'e}gularisants locaux pour des
  {\'e}quations dispersives g{\'e}n{\'e}rales}, Comptes rendus de
  l'Acad{\'e}mie des sciences. S{\'e}rie 1, Math{\'e}matique \textbf{304}
  (1987), no.~14, 407--410.

\bibitem{dBD99}
A.~de~Bouard and A.~Debussche, \emph{A stochastic nonlinear {S}chr\"{o}dinger
  equation with multiplicative noise}, Comm. Math. Phys. \textbf{205} (1999).

\bibitem{dBD03}
\bysame, \emph{The stochastic nonlinear {S}chr\"{o}dinger equation in {$H^1$}},
  Stochastic Anal. Appl. \textbf{21} (2003), no.~1, 97--126. \MR{1954077}

\bibitem{Dodson3}
Benjamin Dodson, \emph{Global well-posedness and scattering for the defocusing,
  {$L^{2}$}-critical nonlinear {S}chr\"{o}dinger equation when {$d\geq3$}}, J.
  Amer. Math. Soc. \textbf{25} (2012), no.~2, 429--463. \MR{2869023}

\bibitem{Dodson1}
\bysame, \emph{Global well-posedness and scattering for the defocusing, {$L^2$}
  critical, nonlinear {S}chr\"{o}dinger equation when {$d=1$}}, Amer. J. Math.
  \textbf{138} (2016), no.~2, 531--569. \MR{3483476}

\bibitem{Dodson2}
\bysame, \emph{Global well-posedness and scattering for the defocusing,
  {$L^2$}-critical, nonlinear {S}chr\"{o}dinger equation when {$d=2$}}, Duke
  Math. J. \textbf{165} (2016), no.~18, 3435--3516. \MR{3577369}

\bibitem{fan2018global}
Chenjie Fan and Weijun Xu, \emph{Global well-posedness for the defocusing
  mass-critical stochastic nonlinear schr{\"o}dinger equation on $\mathbb{R}$
  at ${L}^{2} $ regularity}, to apper in Analyis $\&$ PDE.

\bibitem{FX}
\bysame, \emph{Decay of the stochastic linear schr\"{o}dinger equation in
  $d\geq 3$ with small multiplicative noise}, Stochastic Partial Differential
  Equations: Analysis and Computations (2020).

\bibitem{fan2019wong}
\bysame, \emph{A wong-zakai theorem for mass critical {NLS}}, SIAM Journal of
  Mathematical Analysis (2021).

\bibitem{fan2020long}
Chenjie Fan and Zehua Zhao, \emph{On long time behavior for stochastic
  nonlinear {S}chr{" o}dinger equations with a multiplicative noise}, arXiv
  preprint arXiv:2010.11045 (2020).

\bibitem{HRZ}
Sebastian Herr, Michael R\"{o}ckner, and Deng Zhang, \emph{Scattering for
  stochastic nonlinear {S}chr\"{o}dinger equations}, Comm. Math. Phys.
  \textbf{368} (2019).

\bibitem{Hor}
Fabian Hornung, \emph{The nonlinear stochastic {S}chr\"{o}dinger equation via
  stochastic {S}trichartz estimates}, J. Evol. Equ. \textbf{18} (2018), no.~3,
  1085--1114.

\bibitem{JSS}
J.-L. Journ\'{e}, A.~Soffer, and C.~D. Sogge, \emph{Decay estimates for
  {S}chr\"{o}dinger operators}, Comm. Pure Appl. Math. \textbf{44} (1991),
  no.~5, 573--604. \MR{1105875}

\bibitem{keel1998endpoint}
Markus Keel and Terence Tao, \emph{Endpoint strichartz estimates}, American
  Journal of Mathematics (1998), 955--980.

\bibitem{Schlagsurvey}
W.~Schlag, \emph{Dispersive estimates for {S}chr\"{o}dinger operators: a
  survey}, Mathematical aspects of nonlinear dispersive equations, Ann. of
  Math. Stud., vol. 163, Princeton Univ. Press, Princeton, NJ, 2007,
  pp.~255--285. \MR{2333215}

\bibitem{sjolin1987regularity}
Per Sj{\"o}lin, \emph{Regularity of solutions to the schr{\"o}dinger equation},
  Duke mathematical journal \textbf{55} (1987), no.~3, 699--715.

\bibitem{tao2006nonlinear}
Terence Tao, \emph{Nonlinear dispersive equations: local and global analysis},
  vol. 106, American Mathematical Soc., 2006.

\bibitem{vega1988schrodinger}
Luis Vega, \emph{Schr{\"o}dinger equations: pointwise convergence to the
  initial data}, Proceedings of the American Mathematical Society \textbf{102}
  (1988), no.~4, 874--878.

\bibitem{WXP}
Xue~Ping Wang, \emph{Time-decay of semigroups generated by dissipative
  {S}chr\"{o}dinger operators}, J. Differential Equations \textbf{253} (2012),
  no.~12, 3523--3542. \MR{2981264}

\bibitem{DZhang}
D.~Zhang, \emph{Stochastic nonlinear schr\"odinger equations in the defocusing
  mass and energy critical cases}, arXiv preprint arXiv:1811.00167 (2018).

\end{thebibliography}

\end{document}